\documentclass[leqno]{amsart}
\usepackage{amsfonts,amssymb,amsmath,amsgen,amsthm}
\usepackage{graphics}
\usepackage{hyperref}
\usepackage{color}
\newcommand{\msc}[2][2000]{%
  \let\@oldtitle\@title%
  \gdef\@title{\@oldtitle\footnotetext{#1 \emph{Mathematics subject
        classification.} #2}}% 
}

\theoremstyle{plain}
\newtheorem{theorem}{Theorem} [section]
\newtheorem{definition}[theorem]{Definition}

\newtheorem{lemma}[theorem]{Lemma}

\newtheorem{proposition}[theorem]{Proposition}

\theoremstyle{remark}
\newtheorem{remark}[theorem]{Remark}
\newtheorem*{notation}{Notation}

\def\C{{\mathbb C}}% complex numbers
\def\R{{\mathbb R}}% real numbers
\def\N{{\mathbb N}}% nonnegative integers
\def\Z{{\mathbb Z}}% integers
\def\T{{\mathbb T}}% torus
\def\Sch{{\mathcal S}}% Schwartz space
\def\O{\mathcal O}
\def\F{\mathcal F}

\def\E{\mathcal E}
\def\u{\mathbf u}
\def\dd{\mathrm d}

\def\({\left(}
\def\){\right)}
\def\<{\left\langle}
\def\>{\right\rangle}
\def\le{\leqslant}
\def\ge{\geqslant}

\def\Eq#1#2{\mathop{\sim}\limits_{#1\rightarrow#2}}
\def\Tend#1#2{\mathop{\longrightarrow}\limits_{#1\rightarrow#2}}

\def\d{{\partial}}
\def\eps{\varepsilon}

\def\si{{\sigma}}

\DeclareMathOperator{\RE}{Re}
\DeclareMathOperator{\IM}{Im}

\numberwithin{equation}{section}

\begin{document}

\title[Frequency cascade in NLS]{A toy model for frequency cascade in
  the nonlinear Schr\"odinger equation}
\author[R. Carles]{R\'emi Carles}
\author[E. Faou]{Erwan Faou}
\address{Univ Rennes, CNRS\\ IRMAR - UMR
  6625\\ F-35000 Rennes, France}
\email{Remi.Carles@math.cnrs.fr}
\email{Erwan.Faou@inria.fr}

\begin{abstract}
  We present an elementary approach to observe frequency cascade on
  forced nonlinear Schr\"odinger equations. The forcing term
  (which
  may also appear as a potential term instead)  consists
  of a constant term, perturbed by a modulated Gaussian well. Algebraic
  computations provide an explicit frequency cascade when time and
  space derivatives are discarded from the nonlinear Schr\"odinger
  equation. We prove stability results, showing that when
  derivatives are incorporated in the model, the initial algebraic
  solution 
  may be little affected, possibly over long time intervals. Numerical
  simulations are provided, which support the analysis. 
\end{abstract}
\thanks{EF is supported by the
Simons Collaboration Grant on Wave Turbulence. A CC-BY public
copyright license has been applied by the authors to the present
document and will be applied to all subsequent versions up to the
Author Accepted Manuscript arising from this submission. }  

\maketitle

\section{Introduction}
\label{sec:intro}

\subsection{Setting}
\label{sec:setting}
The phenomenon of frequency cascade originates from the theory of
turbulence in fluid mechanics
\cite{K1,K2,K3} and in wave turbulence see
e.g. \cite{Nazarenko,Zakharov84,Zakharov92}.  
A paradigmatic description in experimental physics can be viewed as a
nonlinear system (typically a wave like system exemplified by the
nonlinear Schr\"odinger equation) forced in a low frequency range,
giving rise to some distribution of energy following some power law of
the form $|\xi|^{-\nu}$. These power law rules can be explained by
many physical arguments involving statistical models, random forcing,
kinetic limit, etc. 

The mathematical approach to 
justify the derivation of equations describing turbulent phenomena
has been very intense, involving very sophisticated tools and
analysis.
The general stragegy is to derive a domain a validity for a wave
kinetic equation possessing power law solution. The analysis can be
performed by considering random initial conditions for a dispersive
equation 
set on a large torus \cite{MMT97,BGHS21,DengHani23,CollotGermain25}, 
leading to resonance analysis and continuous limit arguments. In a
parallel approach, random forcing can be considered,  
\cite{Kuksin97,Faou20,GrandeHani-p}, as well as model on the whole space 
\cite{ACG25,HRST-p,HSZ24,FaouMouzard24}, with specific initial
conditions.

Our goal here is to propose an elementary model, certainly too simple to
meet the actual physics described in the above references, for which
the frequency 
cascade phenomenon is rather explicit. We retain the idea of forcing a
wave system at low frequency and recovering a power law energy
distribution as stationary solution.  
More precisely, we consider the nonlinear Schr\"odinger equation on $\R^d$ (or on the torus $\T^d$, see Remark \ref{remarktorus}), with an
external forcing $f$,
\begin{equation}\label{eq:NLS}
  i\d_t u +\eps^2\Delta u = |u|^{2\si}u - f,\quad x\in \R^d, 
\end{equation}
where $\si>0$, and we ask the following question:
from a smooth forcing term $f$ 
with Fourier transform $\hat f(\xi)$ localized around $0$, can we
  obtain singular or quasi-singular solutions $u$ exhibiting power law
  frequency $\hat u (\xi) \sim |\xi|^{-\nu}$ for some $\nu$, at least
  in some frequency range?  

Note that such a phenomenon must be  nonlinear, as this cannot hold
for the linear equation without the term $|u|^{2\si}u$. Moreover, if
$f$ is smooth, standard existence results show that the solution to
\eqref{eq:NLS} remains smooth in general, so a solution cannot exhibit
a frequency behavior of the form $|\xi|^{-\nu}$ for \emph{all} $\xi
\in \R^d$, 
but only on some (possibly large) range of frequency. This leads us to the notion
of quasi-singular solution, {\em i.e.} we propose a model where
solutions can be of that type.

We consider two cases of forcing term $f$ in \eqref{eq:NLS}:

\begin{itemize}
\item Time independent forcing, $f=f(x)$,
\item Time periodic forcing, $f(t,x)= 2Q(x) e^{-3iPt}$, where $P$ is
  a constant, in the case of a cubic nonlinearity $\si=1$.
\end{itemize}
We do not consider the presence of other dissipative terms in the
general presentation, but we could add a term of the form
$i\nu\eps^2\Delta u$ ($\nu>0$) on the right hand side of
\eqref{eq:NLS}. The stability analysis stated in
Theorem~\ref{theo:stab-strichartz} would remain valid (and could be
possibly strengthened thanks to dissipative effects), see also
Remark~\ref{rem:damping}.
\smallbreak

In many situations, the starting point to analyze systems of the form
\eqref{eq:NLS} is the linear equation, or a 
specific solution of the stationary Schr\"odinger equation (like solitons), from which
perturbative arguments are developed. Our approach is somehow the
opposite, in the sense that we first remove the differential operators
from \eqref{eq:NLS}. We obtain solutions by algebraic means, where singularities or quasi-singularities of Puiseux type $z^{1/p}$ may naturally develop,  and then
examine their stability when the differential operators are
incorporated.

\subsection{Cascade cartoon}
\label{sec:cartoon}

When $f$ does not depend on time, we consider, for $x\in \R^d$,
$\delta\ge 0$, the  forcing 
\begin{equation*}
  f (x) = 1-e^{-\delta} e^{-|x|^2/2}.
\end{equation*}
The algebraic solution $U$ that we consider solves
\begin{equation*}
  |U|^{2\si}U = f.
\end{equation*}
We first remark that $U$ can be taken independent of time, and that $f$ being positive,
the above equation has at least one real-valued solution, given by
\begin{equation*}
  U(x) = f(x)^{\frac{1}{2\si+1}} \in \R. 
\end{equation*}
We note that since $\si>0$, the power
  $\frac{1}{2\si+1}$ lies in $(0,1)$.
If we assume $\delta=0$, we can write
\begin{equation}\label{eq:factor-sing}
  f^{\frac{1}{2\si+1}}-1 = |x|^{\frac{2}{2\si+1}}\psi(x),
\end{equation}
where $\psi$ is smooth everywhere, and rapidly decaying at
infinity.
We normalize the Fourier transform as
\begin{equation*}
  \F
  \psi(\xi)=\hat\psi(\xi):=\frac{1}{(2\pi)^{d/2}}\int_{\R^d}e^{-ix\cdot
    \xi}\psi(x)\dd x,\quad \psi\in \Sch(\R^d),
\end{equation*}
so that Plancherel formula reads $\|\hat\psi\|_{L^2(\R^d)} =
\|\psi\|_{L^2(\R^d)}$. 
Taking the Fourier transform in \eqref{eq:factor-sing} (as 
tempered distributions), we can write
\begin{equation*}
  \F\( f^{\frac{1}{2\si+1}}-1 \) = \F\(
  |x|^{\frac{2}{2\si+1}}\)\ast\F(\psi). 
\end{equation*}
In view of e.g. \cite[Theorem~7.1.16]{Hormander1}, $\F\(
  |x|^{\frac{2}{2\si+1}}\)$ is
  $\(-d-\frac{2}{2\si+1}\)$-homogeneous. Since the map $x\mapsto
  |x|^{\frac{2}{2\si+1}}$ 
  is radially symmetric, so is its Fourier transform, and
\begin{equation*}
  \F \( |x|^{\frac{2}{2\si+1}}\) (\xi) = c(d,\si)
  |\xi|^{-d-\frac{2}{2\si+1}},
\end{equation*}
which enjoys an explicit frequency decay. The idea now is that for
$\delta>0$,  a similar cascade remains, up to restricting the range of
frequencies, and that the solution thus obtained is sufficiently
stable so the phenomenon remains when going back to the nonlinear
Schr\"odinger equation \eqref{eq:NLS}.

\subsection{Properties of the algebraic solution}
\label{sec:res1}

The rigorous meaning of the comparison symbols used in this
article is as follows:
\begin{notation}
  For $a,b\ge 0$, we write $a\approx b$ if there exists $C>1$ such
  that
  \begin{equation*}
    \frac{1}{C}a\le b\le C a.
  \end{equation*}
  We write $a\ll b$ if there exists $\eta\in (0,1)$ (small enough)
  such that
  \begin{equation*}
    a\le \eta b.
  \end{equation*}
\end{notation}
Our first result concerns the algebraic solution, denoted by $U$
above: we check that  an explicit  frequency cascade occurs, on a
frequency range depending on the parameter $\delta$. 
\begin{theorem}\label{theo:law-f}
  Let $d\ge 1$, and for $\delta>0$, consider
 \begin{equation*}
  f (x) = 1-e^{-\delta} e^{-|x|^2/2}.
\end{equation*}
Let $\alpha\in (0,1)$. Then the Fourier transform of $f^\alpha-1$
satisfies
\begin{equation*}
 \left| \F\(f^\alpha-1\)(\xi)\right| \approx |\xi|^{-2\alpha-d},
\end{equation*}
provided that $\xi$ is restricted to a region of the form
\begin{equation*}
 0<\xi_0^2\le |\xi|^2\ll \frac{1}{\delta},
\end{equation*}
for $\xi_0>0$ arbitrary. 
\end{theorem}
\begin{remark}
\label{remarktorus}[Spatially periodic case]
  In the case where $x$ belongs to the circle, $x\in \T=\R/(2\pi\Z)$,
  the above  result can be adapted, by considering
  \begin{equation*}
    {\mathbf f}(x) = 1- e^{-\delta}e^{\cos x -1}.
  \end{equation*}
  Computations are similar to those presented in the proof of
  Theorem~\ref{theo:law-f}, up to the introduction of modified
  Bessel functions of the first kind, and, again, the use of Stirling
  formula. Detailed computations are cumbersome though, and we leave
  out this framework here. We simply note that when the radius of the
  circle becomes unbounded, a natural connection with the model
  considered on $\R$ appears: for $x\in \T_L=\R/(2\pi L\Z)$, we may set
  \begin{equation*}
    {\mathbf f}_L(x) = 1- e^{-\delta}e^{L^2\(\cos (x/L) -1\)},
  \end{equation*}
  and we check the pointwise convergence
  \begin{equation*}
    {\mathbf f}_L(x)\Tend L \infty 1-e^{-\delta}e^{-x^2/2}=f(x).
  \end{equation*}
\end{remark}
In the case where the forcing term is periodic in time ($P\not =0$),
we choose a profile $Q$ which may seem surprising at first sight, but
which is constructed in order to make a connection with the previous
result (see Section~\ref{sec:time-periodic} for details).
\begin{theorem}\label{theo:cubic}
  Let $d\ge 1$, $\si=1$, and $\eps=0$ in \eqref{eq:NLS}. Setting
  $f(t,x) = 2Q(x)e^{-3it}$ with
  \begin{equation*}
    Q(x) =  4 e^{- \frac{3\delta}{2} - \frac{3|x|^2}{4}}  - 3 e^{-
  \frac{\delta}{2} - \frac{|x|^2}{4}},\quad \delta>0,
  \end{equation*}
  \eqref{eq:NLS} has three solutions of the form
  $u_k(t,x)=U_k(x)e^{-3it}$ with $U_k$ real-valued, $k=0,1,2$. Two of
  these three profiles enjoy the same frequency cascade as in
  Theorem~\ref{theo:law-f} (with $\alpha=\frac{1}{2}$), say $U_1$
  and $U_2$:
  \begin{equation*}
    \F(U_k)(\xi) \approx |\xi|^{-1-d}, \quad k=1,2,
  \end{equation*}
provided that $\xi$ is restricted to a region of the form
\begin{equation*}
 0<\xi_0^2\le |\xi|^2\ll \frac{1}{\delta},
\end{equation*}
for $\xi_0>0$ arbitrary.
\end{theorem}

\subsection{Stability}
\label{sec:stab-intro}

Once the algebraic solution is constructed, and its frequency cascade
is analyzed, we have to go back to the nonlinear Schr\"odinger
equation \eqref{eq:NLS}. It turns out that the presence of the
parameter $\delta$ introduced above is crucial in order not to destroy
the frequency cascade, when at least the Laplacian term $\eps^2\Delta$
is present.
\smallbreak

For rather general $\si$ and $d$, we show that if $\eps>0$ is
sufficiently small (compared to $\delta>0$), then incorporating the
Laplacian term from \eqref{eq:NLS} into the algebraic equation does
not destroy the initial structure, to construct a stationary
solution. At this stage, we need a smooth
  nonlinearity, since we will consider higher order derivatives, hence
  the assumption that $\si$ is an integer in the next result.

\begin{theorem}\label{theo:constr-stat}
  Let $d\ge 1$. Suppose that
  \begin{itemize}
  \item Either $\si>0$ is an integer, $\alpha= \frac{1}{2\si+1}$, $f=:2Q$ is like in
    Theorem~\ref{theo:law-f}, $P=0$, and
    $u_0=f^{\alpha}$,  
  \item Or $P=\si=1$, $\alpha=\frac{1}{2}$, $Q$ is like in
  Theorem~\ref{theo:cubic}, and $u_0=U_1$ (see Section~\ref{sec:derivationQ} for the precise definition of $U_1$). 
  \end{itemize}
  Then there exists $K>0$ such that if  $0<\eps\ll \delta^K\ll 1$, there exists a
  real-valued solution $u_\eps= u_0+v$ to
  \begin{equation}\label{eq:NLSstat-gen}
  -\eps^2\Delta u_\eps -3Pu_\eps+ |u_\eps|^{2\si}u_\eps =2Q,
\end{equation}
such that
\begin{equation*}
  \sup_{\xi\in \R^d}\left| |\xi|^{2\alpha+d}\(\widehat u_\eps(\xi) -
    \widehat u_0(\xi)\)\right|\lesssim \frac{\eps}{\delta^K}\ll 1. 
\end{equation*}
In particular, $u_\eps$ enjoys the same frequency cascade as
$u_0$, as described in Theorems~\ref{theo:law-f} and \ref{theo:cubic},
respectively. 
\end{theorem}
\begin{remark}[Focusing case]\label{rem:focusing}
If instead of \eqref{eq:NLS}, where the nonlinearity is defocusing, we
consider
\begin{equation}
  \label{eq:focusing}
  i\d_t u +\eps^2\Delta u = -|u|^{2\si}u+f,\quad x\in \R^d,
\end{equation}
then Theorems~\ref{theo:law-f} and \ref{theo:stab-strichartz}
(when $u_0=f^{\frac{1}{2\si+1}}$) remain unchanged, as will be clear
from the proofs. In the case where $f$ is periodic in time, $f(t,x) =
2Q(x)e^{-3iPt}$, then the only difference in Theorem~\ref{theo:cubic}
is that $P=1$ has to be replaced by $P=-1$, and then, with this
modification, the whole of Theorem~\ref{theo:stab-strichartz} remains valid.
The main difference concerns Theorem~\ref{theo:constr-stat}, since
\eqref{eq:NLSstat-gen} is replaced by (recall that $P=-1$ now, and $\si=1$)
\begin{equation}
  \label{eq:NLSstat-foc}
  -\eps^2\Delta u_\eps +3u_\eps- |u_\eps|^{2}u_\eps =-2Q.
\end{equation}
Linearizing the left hand side about $u_0$, we find the operator
$-\eps^2\Delta +3 -3 U_k^2$, where for $k=0,1,2$, $U_k$ are the
solutions provided by Theorem~\ref{theo:cubic}. Unlike what happens in
the defocusing case, this operator is elliptic provided that $U_k^2$
remains uniformly smaller than $1$. Resuming the proof of
Theorem~\ref{theo:constr-stat} then shows for $k=0,1,2$, each $U_k^2$
takes larger values than $1$, and  Theorem~\ref{theo:cubic} may not be
true in the focusing case.
\end{remark}

When the nonlinearity is smooth ($\si\in \N$) and
$L^2$-subcritical or critical (like in the statement below),
Strichartz estimates make it possible 
to prove the following dynamical stability result:

\begin{theorem}\label{theo:stab-strichartz}
  Let $d=1$ with $\si=1$ or $\si=2$, or $d=2$ with $\si=1$. Denote by
  $\u_0=f^{\frac{1}{2\si+1}}$ when $f$ is like in
  Theorem~\ref{theo:law-f}, and $\u_0=U_k(x)e^{-3it}$ ($k=2,3$)
  when $\si=1$ and $f$ is like in Theorem~\ref{theo:cubic}. Suppose
  that $u$ solves \eqref{eq:NLS} with
  \begin{equation}\label{eq:CI}
    u_{\mid t=0} = \u_{0\mid t=0}+v_0.
  \end{equation}
  If $v_0\in L^2(\R^d)$ is such that
  \begin{equation*}
    \|v_0\|_{L^2(\R^d)}\ll \eps^{d/2}, 
  \end{equation*}
  and if
  \begin{equation*}
    \eps^{2-\frac{d}{2}}\ll \delta^{\frac{d}{4}+1-\frac{1}{2\si+1}},
  \end{equation*}
  then there exist $T>0$ and a unique solution $u=\u_0+v$ to
  \eqref{eq:NLS}--\eqref{eq:CI}, such that 
  \begin{equation*}
    v\in C([0,T];L^2(\R^d))\cap
    L^{\frac{4\si+4}{d}}([0,T];L^{2\si+2}(\R^d)). 
  \end{equation*}
  In addition, there exists $C>0$ such that
  \begin{equation*}
    \|v(t)\|_{L^2(\R^d)}\le C\eps^{d/2}e^{Ct},\quad 0\le t\le T . 
  \end{equation*}
  If there exists $\iota,C>0$ independent of $\eps$ such that
  \begin{equation*}
    \|v_0\|_{L^2(\R^d)}\le C \eps^{d/2+\iota}\quad\text{and}\quad
    \eps^{2-\frac{d}{2}-\iota}\ll \delta^{\frac{d}{4}+1-\frac{1}{2\si+1}}, 
  \end{equation*}
  then there exist $\eps_0\in (0,1)$ and  $c_1$ independent of
  $\eps\in (0,\eps_0)$ such that the above conclusions remain  valid
  with $T=T^\eps=c_1\log\eps^{-1}$. 
\end{theorem}
\begin{remark}[Potential driven cascade]\label{rem:potential} 
  If instead of a source term,
  we set (the same) $f$ as a potential, and consider
\begin{equation}\label{eq:potential}
  i\d_t u +\eps^2\Delta u = |u|^{2\si}u - fu ,\quad x\in \R^d,
\end{equation}
 then our algebraic solution is still given like in
 Theorem~\ref{theo:law-f}, with $\alpha=\frac{1}{2\si}$ instead of
 $\alpha=\frac{1}{2\si+1}$.  Up to this change for the value of
 $\alpha$ and minor modifications, Theorem~\ref{theo:constr-stat} (first case), and
 Theorem~\ref{theo:stab-strichartz} remain valid. Note that we may
 also replace $f$ with $f+\omega$ for some $\omega\in \R$: $u_0(t,x)
 :=e^{i\omega t}U(x)$ solves \eqref{eq:potential} in the case
 $\eps=0$, where $U= f^{1/(2\si)}$ exhibits a frequency cascade, and
 this solution is stable in the sense of
 Theorem~\ref{theo:stab-strichartz}. 
\end{remark}

\begin{remark}[Fluid mechanics]\label{rem:fluid}
  Developping on the previous remark, we may consider the Madelung
  transform: $\rho:=|u|^2$ and $J=\IM \(\bar u \nabla u\)$ solve the
  Euler-Korteweg (or quantum Euler) equation
  \begin{equation*}
    \left\{
      \begin{aligned}
        &\d_t \rho +2\eps^2 \nabla\cdot J=0,\\
        &\d_t J +2\eps^2 \nabla\(\frac{J\otimes J}{\rho}\) +
        \frac{\si}{\si+1}\nabla \rho^{\si+1} -\rho \nabla f =
        \frac{\eps^2}{2}\Delta\nabla \rho - 2\eps^2\nabla\cdot\( \nabla
        \sqrt\rho\otimes \nabla
        \sqrt\rho\),
      \end{aligned}
    \right.
  \end{equation*}
  see e.g. \cite{CaDaSa12}. The algebraic solution becomes $\rho_0=
  f^{1/\si}$: if $\si>1$, then we can still invoke
  Theorem~\ref{theo:law-f} to describe a frequency cascade for
  $\rho_0$. On the other hand, there is no quasi-singularity for
  $\rho_0$ when $\si=1$. 
\end{remark}

The rest of the paper is organized as follows. In Section~\ref{sec:f},
we prove Theorem~\ref{theo:law-f}. It paves the way for the proof
of Theorem~\ref{theo:cubic}, which is given in
Section~\ref{sec:time-periodic}. The stability analysis, leading to
Theorems~\ref{theo:constr-stat} and \ref{theo:stab-strichartz}, is
presented in Section~\ref{sec:stability}. Finally, numerical simulations
illustrate and complement the analysis in Section~\ref{sec:num}.
\section{Time independent forcing}
\label{sec:f}

In this section, we consider
\begin{equation*}
    f (x) = 1-e^{-\delta} e^{-\beta|x|^2/2},\quad \beta,\delta>0.
  \end{equation*}
The parameter $\delta$ is crucial in the sequel, while $\beta$ is
here mostly to show some form of generality (see also Remark~\ref{rem:other}).
  
To lighten notations, we introduce $\alpha \in (0,1)$, and consider the
function $f^\alpha-1$. Of course, $\alpha$ has to be thought of as
$\alpha=\tfrac{1}{2\si+1}$. The function $f^\alpha-1$ is given by a
rather explicit power series. Indeed, for $|X|<1$, 
\begin{equation*}
 (1-X)^{\alpha}-1 = \sum_{n\ge 1} a_n (-X)^n,
\end{equation*}
where the formula for the $a_n$'s is standard,
\begin{equation*}
  a_n =
  \frac{\alpha(\alpha-1)\dots(\alpha-n+1)}{n!}=
  \frac{\Gamma(\alpha+1)}{\Gamma(n+1) 
    \Gamma(\alpha-n+1)}. 
\end{equation*}
Euler formula reads, for $z\in \C\setminus \Z$,
\begin{equation*}
  \Gamma(1-z)\Gamma(z) = \frac{\pi}{\sin(\pi z)},
\end{equation*}
so we can rewrite ($z= \alpha-n+1\not\in \Z$ indeed)
\begin{equation*}
  a_n =
  \frac{\Gamma(\alpha+1)}{\Gamma(n+1)}
  \frac{\sin\(\pi(\alpha-n+1)\)}{\pi}
  \Gamma(n-\alpha)= (-1)^{n-1} \frac{\sin\(\pi\alpha\)}{\pi}
 \frac{\Gamma(\alpha+1)\Gamma(n-\alpha)}{\Gamma(n+1)} .
\end{equation*}
Stirling formula yields
\begin{equation*}
  \Gamma(z)\Eq z \infty \sqrt{2\pi} z^{z-\frac{1}{2}}e^{-z}.
\end{equation*}
Therefore
\begin{equation*}
   \frac{\Gamma(n-\alpha)}{\Gamma(n+1)}\Eq n \infty \frac{c}{n^{\alpha+1}},
 \end{equation*}
 for some $c>0$ whose exact value can be computed, but which is left
 out here.
 Since we always have $\alpha>0$, the series is normally
 converging, and
 \begin{equation}\label{eq:f-alpha-1}
   \begin{aligned}
     \mathcal{F}( f^\alpha-1)(\xi)
     &= \sum_{n\ge 1} (-1)^n a_n \F\( e^{-n\delta}
       e^{-n\beta|x|^2/2}\)(\xi) \\
     &=\sum_{n\ge 1} (-1)^n a_n\frac{1}{(\beta n)^{d/2}}e^{-n\delta}
   e^{-\frac{|\xi|^2}{2n\beta}}.
\end{aligned}
\end{equation}
In view of the above computations,
\begin{equation*}
 (-1)^n a_n\Eq n \infty - \frac{\sin\(\pi\alpha\)}{\pi}
 \Gamma(\alpha+1) \frac{c}{n^{\alpha+1}}= \frac{c_0}{n^{\alpha+1}},
\end{equation*}
for some $c_0<0$ whose precise value is irrelevant for our purpose. 
\smallbreak

The sum on $n$ is then compared to the continuous counterpart via an
integral. This makes sense when the integrated function in
monotonic. In the present case, and after replacing $a_n$ with its
leading order asymptotic behavior, the function we consider is
\begin{equation*}
  g(y) = \frac{1}{y^{\alpha+1+d/2}}
    e^{-y\frac{\delta}{\beta}}e^{-\frac{|\xi|^2}{2y}},
  \end{equation*}
where $y$ is a placeholder for $n\beta$.
Indeed, there exist $c_0,C_0>0$ such that 
\begin{equation*}
  \frac{c_0}{n^{\alpha+1}}\le (-1)^{n}a_n\le
  \frac{C_0}{n^{\alpha+1}},\quad \forall n\ge 1.
\end{equation*}

Let
\begin{equation*}
  k= \alpha+1+d/2,\quad \eta = \frac{\delta}{\beta},\quad \gamma =
  \frac{|\xi|^2}{2}.
\end{equation*}
Then
\begin{equation*}
  g(y) = \frac{1}{y^k}e^{-\eta y-\frac{\gamma}{y}},\quad g'(y) =
  \frac{1}{y^k}e^{-\eta y-\frac{\gamma}{y}}\( -\frac{k}{y}-\eta+\frac{\gamma}{y^2}\).
\end{equation*}
The sign of $g'$ is given by the second order polynomial in $1/y$,
\begin{equation*}
  P(X) = \gamma X^2 -kX-\eta.
\end{equation*}
The function $g$ is nonincreasing for
\begin{equation*}
  \frac{1}{y}= X<\frac{k+\sqrt \Delta}{2\gamma},\quad \Delta = k^2+4\eta\gamma.
\end{equation*}
In the range $\frac{\delta}{\beta}|\xi|^2\ll 1$, $\Delta\approx k^2$, and the above
condition reduces to
\begin{equation*}
  y>\frac{\gamma}{k} = \frac{|\xi|^2}{2\alpha+2+d}=:y_0.
\end{equation*}
Therefore, $g$ is increasing on $(0,y_0)$, decreasing on
$(y_0,\infty)$. Since for all $a>0$, the map $z\mapsto  z^a e^{-z}$ is bounded on
$[0,\infty)$, the map
\begin{equation*}
  y\mapsto \(\frac{\gamma}{y}\)^a e^{-\frac{\gamma}{y}}
\end{equation*}
is bounded on $(0,\infty)$, independently of $\gamma>0$. Writing
\begin{equation*}
  g'(y) = -\frac{\eta}{\gamma^k}\(\frac{\gamma}{y}\)^ke^{-\eta
    y-\frac{\gamma}{y}}
  -\frac{k}{\gamma^{k+1}}\(\frac{\gamma}{y}\)^{k+1}e^{-\eta
    y-\frac{\gamma}{y}} + \frac{1}{\gamma^{k+1}}\(\frac{\gamma}{y}\)^{k+2}e^{-\eta
    y-\frac{\gamma}{y}},
\end{equation*}
we get a dependence of bound upon $\xi$ of the form
\begin{equation*}
  |g'(y)|\lesssim \frac{1}{|\xi|^{2k+2}} + \frac{1}{|\xi|^{2}}.
\end{equation*}
In passing, this shows that $g$ is integrable near the origin. 
So if we assume that $|\xi|\ge \xi_0>0$, $|g'|$ is controlled on
$(0,\infty)$ by a constant which depends on $\xi_0, \alpha$ and $d$
(only). 
Therefore,
\begin{itemize}
\item As $g'>0$ on $(0,y_0)$, there exists $c_1,C_1$ (independent of
  $|\xi|^2\ge \xi_0^2$)
  such that
  \begin{align*}
    c_1\int_0^{y_0}\frac{1}{y^{\alpha+1+d/2}} e^{- y\frac{\delta}{\beta}}
    e^{ -  \frac{|\xi|^2}{2y}}\dd y
    &\le \sum_{n\beta \le y_0}
     \frac{(-1)^{n-1} }{\(\beta n\)^{d/2}}a_n e^{-
      n \delta} e^{ - \frac{|\xi|^2}{2n \beta}}\\
    &\le
     C_1\int_0^{y_0}\frac{1}{y^{\alpha+1+d/2}} e^{- y\frac{\delta}{\beta}}
    e^{ -  \frac{|\xi|^2}{2y}}\dd y,
  \end{align*}
 \item As $g'<0$ on $(y_0,\infty)$, there exists $c_2,C_2$ (independent of
  $|\xi|^2$ since $\|g'\|_{L^\infty}$ is independent of $|\xi|^2\ge \xi_0^2$)
  such that
   \begin{align*}
    c_2\int_{y_0}^\infty\frac{1}{y^{\alpha+1+d/2}} e^{- y\frac{\delta}{\beta}}
    e^{ -  \frac{|\xi|^2}{2y}}\dd y
    &\le \sum_{n\beta > y_0}
     \frac{(-1)^{n-1} }{\(\beta n\)^{d/2}}a_n e^{-
      n \delta} e^{ - \frac{|\xi|^2}{2n \beta}} \\
     &\le
     C_2\int_{y_0}^\infty\frac{1}{y^{\alpha+1+d/2}} e^{- y\frac{\delta}{\beta}}
    e^{ -  \frac{|\xi|^2}{2y}}\dd y.
  \end{align*}
\end{itemize}
Now we set $c_3=\min(c_1,c_2)$ and $C_3=\max(C_1,C_2)$, so we have
   \begin{align*}
    c_3\int_{0}^\infty\frac{1}{y^{\alpha+1+d/2}} e^{- y\frac{\delta}{\beta}}
    e^{ -  \frac{|\xi|^2}{2y}}\dd y
   & \le \sum_{n\ge 1}
     \frac{(-1)^{n-1}}{\(\beta n\)^{d/2}} a_n e^{-
     n \delta} e^{ - \frac{|\xi|^2}{2n \beta}} \\
     &\le
     C_3\int_0^\infty\frac{1}{y^{\alpha+1+d/2}} e^{- y\frac{\delta}{\beta}}
    e^{ -  \frac{|\xi|^2}{2y}}\dd y.
  \end{align*}
So we can take $\xi_0>0$ arbitrary, and under the constraint
\begin{equation*}
  |\xi|^2\ll \frac{\beta}{\delta},
\end{equation*}
we can write
\begin{equation*}
  \int_0^\infty\frac{1}{y^{\alpha+1+d/2}} e^{- y\frac{\delta}{\beta}}
    e^{ -  \frac{|\xi|^2}{2y}}\dd y\approx c\beta^{\alpha}|\xi|^{-2\alpha-d} \int_0^\infty
      z^{\alpha-1+d/2}e^{-z/2}\dd z,
    \end{equation*}
 thus proving the explicit frequency power law, and
 Theorem~\ref{theo:law-f}.

\begin{remark}[Other power for the singularity]\label{rem:other}
  To model a cancellation for $f$ of the form $|x|^{2k}$, with $k>0$
  (integer or not),
  we may consider
  \begin{equation*}
    \(1 - e^{-\delta}e^{ - \beta |x|^2/2}\)^k.
  \end{equation*}
  We can then proceed similarly, by
  simply replacing $\frac{1}{2\si+1}$ by $\frac{k}{2\si+1}$ for the
  value $\alpha$. The case $k<2\si+1$ requires no modification. If
  $k>2\si+1$, and if $\frac{k}{2\si+1}$ is not an integer, we write
  \begin{equation*}
    \frac{k}{2\si+1}= n+\alpha,\quad \alpha\in (0,1), 
  \end{equation*}
and 
\begin{equation*}
    \(1 - e^{-\delta}e^{ - \beta |x|^2/2}\)^k= \(1 - e^{-\delta}e^{ -
      \beta |x|^2/2}\)^n
  \(1 - e^{-\delta}e^{ - \beta |x|^2/2}\)^\alpha.
  \end{equation*}
The first factor on the right hand side is smooth, and the second is
treated like below: only the second factor is responsible for a
quasi-singularity. 
\end{remark}
\section{Time periodic forcing}
\label{sec:time-periodic}

\subsection{Cardano formula}
\label{sec:cardano}

We now assume that 
\begin{equation*}
  f=f(t,x)= 2Q(x) e^{-3iPt}\quad \text{and}\quad \si=1,
\end{equation*}
where no assumption is made so far on the sign of
  $P$. It turns out that the analysis below yields a result compatible
  with the previous  approach ($P=0$) ony for $P>0$.
We somehow proceed in a reversed order compared to the previous
section: we seek an algebraic solution to \eqref{eq:NLS}, and impose
that the profile has a similar behavior as in the previous
section. This will 
lead to the definition of $Q$.

The first step consists in writing the solution $u$ to \eqref{eq:NLS}
under the (approximate) form $u(t,x) = e^{-3iPt}U(x)$, and removing
the Laplacian (we formally let $\eps \to 0$). If $Q$ is real,
then so is $U$, which therefore solves
\begin{equation}\label{eq:algebre}
  U^3 - 3P U - 2Q = 0. 
\end{equation}
The solutions are given by Cardano formula, we briefly recall the
corresponding method.
We plug $U = \alpha + \beta \in \R$ into \eqref{eq:algebre}, and we
find 
\begin{equation*}
  \alpha^3 + \beta^3 + (3 \alpha \beta - 3P)( \alpha + \beta) - 2Q = 0.
\end{equation*}
We impose
\begin{equation*}
\left|
\begin{array}{l}
\alpha \beta = P \quad \Leftrightarrow \quad \alpha^3 \beta^3 = P^3\\
\alpha^3 + \beta^3 = 2Q
\end{array}
\right.
\end{equation*}
and this yields that $\alpha^3$ and $\beta^3$ are the roots of 
\begin{equation*}
  X^2 - 2Q X + P^3=0\quad \Leftrightarrow \quad X = Q \pm \sqrt{Q^2 - P^3}.
\end{equation*}
We discuss this formula according to the sign of the discriminant $Q^2
- P^3$, which depends on the value of $Q(x)$. 

\begin{itemize}
\item $Q^2 > P^3$. 
%Let us recall that $P > 0$ hence we can take
%  $\alpha^3 \in [Q, Q + |Q|]$ and $\beta^3 \in [Q - |Q|,Q]$. 
Note that
  necessarily $\alpha \in \R$ and $\beta \in \R$. In addition,
    $\alpha^3+\beta^3 = 2Q$ and $\alpha+\beta=U$ have the same sign,
    since 
    \begin{equation*}
      \alpha^3+\beta^3 = (\alpha+\beta)\underbrace{(\alpha^2-\alpha
        \beta+\beta^2)}_{\ge 0}.
    \end{equation*}
  \item $Q^2 = P^3$. Then we have $\alpha^3 = \beta^3 = Q$. If we take
  $\alpha = j^k Q^\frac13$ where $Q^\frac13$ denote the real cubic root of
  the real function $Q$, $j = e^{\frac{2i\pi}{3}}$, $k = 0,1,2$, then 
\begin{itemize}
\item If $k = 0$ then $U = 2Q^{\frac13}$ has the same sign as $Q$.  
\item If $k = 1$ or $k = 2$ (the same case by noticing $\bar j = j^2$) then $U = 2Q^{\frac13} \cos(\frac{2\pi}{3}) = - Q^{\frac13}$. We obtain the solution 
\begin{equation*}
%\label{SOL1}
Q^2 = P^3,\quad
U = - Q^{\frac13} .
\end{equation*}
If $P \neq 0$ is constant, this
solution cannot exhibit 
singularities as we have necessarily $P>0$,  $Q(x)= P^\frac{3}{2}$, and thus
$U\sim -\sqrt P$ is smooth near the point $x$ considered.
\end{itemize}
\item $Q^2 < P^3$. Then as $P> 0$ is constant we can write 
  \begin{equation*}
    Q = P^\frac{3}{2} R, \quad \mbox{with}\quad R < 1. 
  \end{equation*}
Then we can take 
\begin{equation*}
\alpha^3 = Q + i\sqrt{P^3 - Q^2} = P^{\frac32} \( R + i \sqrt{1 - R^2}\), \quad 
\mbox{and}\quad \beta = \bar \alpha,
\end{equation*}
with the condition $R^2 \le 1$ and thus $R = \cos(\theta(x))$ for some function 
\begin{equation*}
  \theta(x) = \arccos( R) \in [0,\pi]. 
\end{equation*}
Let us remark that $\theta$ can have jumps in the derivative, but is at least continuous if $R$ is smooth. 

\begin{remark}
At this level, we could also define another lifting as $\theta(x) + 2k(x) \pi$ is also solution, with $k(x) \in \Z$. Then we would have jump in the derivative, but these jumps could be compensated, see below.  
\end{remark}
Note that $\alpha \beta = |\alpha|^2 = P$, and 
and hence $|\alpha| = P^{\frac12}$. 
Then we have 
\begin{equation*}
  \alpha^3 =  P^{\frac32}( \cos \theta + i |\sin \theta|) = P^{\frac32} e^{i \theta},
\end{equation*}
as  $\theta \in [0,\pi]$. Note that the limiting cases $\theta = 0$
and $\theta = \pi$ correspond to the previous case $Q^2 = P^3$. Then
we can take with $k(x) \in \{0,1,2\}$ depending possibly on $x$, 
\begin{equation*}
  \alpha = P^{\frac12} e^{i \frac{\theta(x)}{3} + i\frac{2k(x)\pi}{3}}
  = P^{\frac12} e^{i \phi(x)}. 
\end{equation*}
\end{itemize}
In the sequel, we assume $P=1$ to lighten notations,
  given we have already assumed $P>0$.
\subsection{Derivation of the forcing term}
\label{sec:derivationQ}

We make the above formal discussion more precise, in order to derive
some explicit formula for $Q$. 
In the case where $P=1$ and the discriminant $Q^2-P^3<0$, we have
found that 
\begin{equation*}
  \alpha^3 = Q+i\sqrt{1-Q^2},\quad\beta=\bar\alpha. 
\end{equation*}
Let $j=e^{2i\pi/3}$. We have three solutions, and all of them are
real-valued,
\begin{equation*}
  U_k=j^k\(Q+i\sqrt{1-Q^2}\)^{1/3} + j^{-k} \(Q-i\sqrt{1-Q^2}\)^{1/3}
  , \quad k=0,1,2,
\end{equation*}
where the cubic root is the same for $k=0,1$ and $2$, noticing that
the imaginary part of $Q+i\sqrt{1-Q^2}$ is never zero, provided that
$Q^2<1$ everywhere. More precisely, for $Q(x)\in (0,e^{-\eta}]$ for all $x\in
\R^d$, we can write $Q(x) = \cos\theta(x)$ for some $\theta(x)\in
[\arccos e^{-\eta},\frac{\pi}{2})$. We have
\begin{align*}
  U_k(x) &=  j^k
  e^{i\theta(x)/3} + j^{-k}e^{-i\theta(x)/3}=2\cos\(
           \frac{\theta(x)}{3}+\frac{2k\pi}{3}\)\\
& = 2 \cos\( \frac{\theta(x)}{3}\) \cos\( \frac{2k\pi}{3}\) - 2\sin\(
  \frac{\theta(x)}{3}\) \sin\( \frac{2k\pi}{3}\), 
\end{align*}
and thus
\begin{align*}
&U_0(x) = 2 \cos\( \frac{\theta(x)}{3}\) ,\\
&U_1(x) = - \cos\( \frac{\theta(x)}{3}\) - \sqrt{3}\sin\( \frac{\theta(x)}{3}\),\\
&U_2(x) = - \cos\( \frac{\theta(x)}{3}\) + \sqrt{3}\sin\(
  \frac{\theta(x)}{3}\). 
\end{align*}
To fall back into the situation of the case $P=0$, we assume that
\begin{equation*}
  \sin\( \frac{\theta(x)}{3}\) = \(1 - e^{- \delta -
    \frac{|x|^2}{2}}\)^{\frac12}, 
\quad \cos\( \frac{\theta(x)}{3}\) = e^{- \frac{\delta}{2} -
  \frac{|x|^2}{4}}. 
\end{equation*}
For $k = 1$ or $2$, the sine term is a special case of
Theorem~\ref{theo:law-f}. Classical trigonometric formulas
(special case of 
Chebyshev polynomial) yield
\begin{align*}
Q(x) &= \cos( \theta(x)) = 4 \cos\( \frac{\theta(x)}{3}\)^3 - 3 \cos\( \frac{\theta(x)}{3}\) \\
&=  4 e^{- \frac{3\delta}{2} - \frac{3|x|^2}{4}}  - 3 e^{-
  \frac{\delta}{2} - \frac{|x|^2}{4}}. 
\end{align*}
Together with Theorem~\ref{theo:law-f}, this yields
Theorem~\ref{theo:cubic}.

\section{Stability}
\label{sec:stability}

The starting point in this section is simply the purely algebraic
solution considered in Section~\ref{sec:f} or
Section~\ref{sec:time-periodic}, 
$U$ or $U_k$, $k=1,2$.
Denote by $\u_0=u_0=U(x)$ the solution constructed in Section~\ref{sec:f},
and $\u_0=u_0 (x)e^{-3it}= U_k(x)e^{-3it}$ ($k=1,2$) the solution
constructed in 
Section~\ref{sec:time-periodic} when $\si=1$, that
  is, like in Theorem~\ref{theo:constr-stat}. We set
$\alpha=\frac{1}{2\si+1}$ when $f$ is stationary (like in
Section~\ref{sec:f}), and $\alpha=\frac{1}{2}$  when $f=2Q(x)e^{-3it}$
and the nonlinearity is cubic (like in
Section~\ref{sec:time-periodic}). 
\smallbreak

We start our analysis
with rather straightforward estimates for $u_0-1$:
\begin{lemma}\label{lem:u0Hs}
  Let $s\ge 0$. We have
  \begin{equation*}
     \| u_0-1\|_{\dot H^s} \lesssim 
     \begin{cases}
       1&\text{ if } \alpha >s/2-d/4,\\
\log\frac{1}{\delta} &\text{ if }\alpha = s/2-d/4,\\
\delta^{\alpha-s/2+d/4}&\text{ if }\alpha<s/2-d/4,
     \end{cases}
   \end{equation*}
   and
   \begin{equation*}
     \|\<\cdot\>^s\Delta u_0\|_{L^2}\lesssim
     \delta^{\alpha-1-d/4},\quad\text{where }\<x\>=\sqrt{1+|x|^2}.
   \end{equation*}
\end{lemma}
\begin{proof}
Recalling \eqref{eq:f-alpha-1}, we write
\begin{equation*}
  \| u_0-1\|_{\dot H^s} \lesssim \frac{1}{\beta^{d/2}}\sum_{n\ge 1}
  \frac{e^{-n\delta}}{n^{1+\alpha+d/2}} \left\| |\xi|^s
      e^{-\frac{|\xi|^2}{2n\beta}}\right\|_{L^2(\R^d)} \lesssim
  \beta^{s/2-d/4} \sum_{n\ge 1} \frac{e^{-n\delta}}{n^{1+\alpha-s/2+d/4}}.
\end{equation*}
As $e^{-n\delta}\le 1$, the conclusion is straightforward when
$\alpha>s/2-d/4$. Otherwise, we split the sum into two, $n\le
1/\delta$ and $n>\delta$. For $n\le 
\delta$, the role of the exponential is negligible, and we write
\begin{equation*}
  \sum_{n\le 1/\delta} \frac{e^{-n\delta}}{n^{1+\alpha-s/2+d/4}}
  \le \sum_{n\le 1/\delta} \frac{1}{n^{1+\alpha-s/2+d/4}}\Eq \delta 0
  \int_1^{1/\delta} \frac{\dd y}{y^{1+\alpha-s/2+d/4}}\Eq \delta 0
  \delta^{\alpha-s/2+d/4}, 
\end{equation*}
with the obvious modification in the case $\alpha=s/2-d/4$. The sum
for $n>1/\delta$ can be viewed as 
a Riemann sum,
\begin{align*}
  \sum_{n> 1/\delta} \frac{e^{-n\delta}}{n^{1+\alpha-s/2+d/4}}
  &=
 \delta^{\alpha-s/2+d/4} \times \delta \sum_{n\delta> 1}
    \frac{e^{-n\delta}}{(\delta n)^{1+\alpha-s/2+d/4}}\\
  &\Eq \delta 0
 \delta^{\alpha-s/2+d/4}\int_1^\infty
 \frac{e^{-y}}{y^{1+\alpha-s/2+d/4}}\dd y. 
\end{align*}
The general Sobolev estimates follow. To estimate the momenta of
$\Delta u_0$, we write
\begin{equation*}
  \F\(\Delta u_0\) = \sum_{n\ge 1} (-1)^n a_n\frac{1}{(\beta
    n)^{d/2}}|\xi|^2 e^{-n\delta}e^{-\frac{|\xi|^2}{2n\beta}},
\end{equation*}
so
\begin{equation*}
    \|\<\cdot\>^s\Delta u_0\|_{L^2}\lesssim \sum_{n\ge 1} \frac{|a_n|}{
    n^{d/2}}e^{-n\delta}\left\| |\xi|^2
    e^{-\frac{|\xi|^2}{2n\beta}}\right\|_{H^s}.
\end{equation*}
In view of the estimate (the largest contribution being the $L^2$-norm),
\begin{equation*}
  \left\| |\xi|^2
    e^{-\frac{|\xi|^2}{2n\beta}}\right\|_{H^s}\lesssim n^{1+d/4},
\end{equation*}
we can  resume the above estimate (setting $s=2$ there), and the lemma follows.
\end{proof}

\subsection{Construction of a stationary profile}
\label{sec:constr-stat}
\subsubsection*{Time independent forcing}
We first consider the case where $f$ does not depend on time, like in Section~\ref{sec:f}, and linearize  
\eqref{eq:NLSstat-gen} about the algebraic solution
$f^{\frac{1}{2\si+1}}$. Again, we write
 $u_\eps = u_0 + v$, where $u_0=f^\alpha=f^{\frac{1}{2\si+1}}$,  with now
 the condition $v\in H^1(\R)$ independent of time.  We obtain
\begin{equation*}
  -\eps^2\Delta v +(\si+1)|u_0|^{2\si}v + \si
  u_0^{\si+1}\bar u_0^{\si-1} \bar v = \eps^2\Delta
  u_0+\O\(|v|^2\). 
\end{equation*}
Note that $u_0$ is real-valued, and so is $v$, hence
\begin{equation*}
   -\eps^2\Delta v +(2\si+1)u_0^{2\si}v = -\eps^2\Delta
  u_0+\O\(v^2\). 
\end{equation*}
As $f(x) \ge 1-e^{-\delta}$ and
$0<\alpha<1$, we have\footnote{The inequality $x^\alpha-y^\alpha\le
  (x-y)^\alpha$ for all 
  $0\le y\le x$ is standard and easy to prove.}  $f^\alpha\ge
1-e^{-\alpha\delta}\ge \alpha \delta$, along with $1\ge u_0\ge
\alpha\delta$. 
Lemma~\ref{lem:u0Hs} yields the bound
\begin{equation}\label{eq:Delta-u0-Hs}
   \|\eps^2\Delta u_0\|_{\dot H^s} \lesssim
  \eps^2\beta^{1+s/2-d/4} \delta^{\alpha-1-s/2-d/4}.
\end{equation}
The operator $-\eps^2\Delta + (2\si+1)u_0^{2\si}$ is elliptic,
uniformly in $\eps$,
\begin{equation}
\label{ellipunif}
  \< -\eps^2\Delta  v+ (2\si+1)u_0^{2\si}v,v\> \ge
  (2\si+1)(\alpha\delta)^{\frac{2\si}{2\si +1}}\|v\|_{L^2}^2. 
\end{equation}
We note that for $s>0$, $\F H^s(\R^d)$ is the weighted $L^2$-space
\begin{equation*}
  \F H^s(\R^d)=\left\{\psi\in L^2(\R^d),\ \int_{\R^d}\(1+|x|^2\)^s
  |\psi(x)|^2\dd x<\infty\right\}. 
\end{equation*}
\begin{lemma}\label{lem:LaxMilgram}
Let $\sigma, \beta>0$ and $k\in \N$ be given. There exist a constant
$C$ and $a = a(k,d,\sigma) >0$ such that the following holds.
\begin{itemize}
\item Let $g \in H^k(\R^d)$. Then for all $\eps \in (0,1]$, $\delta>0$ and $u_0$ as
above, there exists a unique solution $v\in H^k(\R^d)$ to the equation
\begin{equation}\label{eq:lin-source}
  -\eps^2\Delta v  + (2\si+1)u_0^{2\si} v = g,
\end{equation}
and moreover, we have the estimate
\begin{equation*}
  \|v\|_{H^k} \le\frac{C}{\delta^{a}} \| g \|_{H^k}. 
\end{equation*}
\item If in addition $g \in \F H^k(\R^d)$, then  $v\in \F H^k(\R^d)$ and
\begin{equation*}
  \|\<\cdot\>^kv\|_{L^2} \le\frac{C}{\delta^{a}} \| \<\cdot\>^kg \|_{L^2}. 
\end{equation*}
\end{itemize}
\end{lemma}
\begin{proof}
We have 
\begin{equation}
  \label{coercivity}
  \begin{aligned}
    \< -\eps^2\Delta  v+ (2\si+1)u_0^{2\si}v,v\>
    &\gtrsim \delta^{\frac{2\si}{2\si+1}}\|v\|_{L^2}^2 + \eps^2 \| v
      \|_{\dot H^1}^2  \\
    &\gtrsim \min\(\delta^{\frac{2\si}{2\si
      +1}},\eps^2\) \| v \|_{H^1}^2.  
  \end{aligned}
 \end{equation}
Thus the Lax-Milgram Theorem guarantees the existence of $v$ in $H^1$. 
The previous bound then shows that 
\begin{equation*}
  \|v \|_{L^2} \le \frac{C}{\delta^{\frac{2\si}{2\si +1}}} \| g \|_{L^2}.
\end{equation*}
In view of the construction of $u_0$ and of Lemma~\ref{lem:u0Hs},
\begin{equation*}
  \| u_0 \|_{\dot H^k} \lesssim  \delta^{\alpha-k/2-d/4},\quad
  \mbox{and}\quad  
 c\delta \le u_0 \le 1. 
\end{equation*}
Then we calculate, with $D = \partial_{x_j}$, for $j\in \{1,\dots,d\}$,
\begin{equation*}
 -\eps^2\Delta (D v)  + (2\si+1)u_0^{2\si} (Dv)  = D g - 2\si(2 \si +1)
 u_0^{2\si-1} (Du_0) v,  
\end{equation*}
from which we deduce 
\begin{align*}
\delta^{2\sigma} \| D u\|_{L^2} &\leq C \| Dg \|_{L^2} + C \| u_0^{2\si-1} (Du_0) \|_{L^\infty} \| v \|_{L^2}\\
& \le C \| Dg \|_{L^2} + C \delta^{-b}\| g \|_{L^2}
\end{align*}
for some constants $b$ and $C$ independent of $\delta$ and
$\varepsilon$. We then deduce the first part of the result by
induction on $k$.
\smallbreak

Suppose that $g\in \F H^k$, and multiply \eqref{eq:lin-source} by
$\<x\>^k$:
\begin{equation*}
  -\eps^2\<x\>^k\Delta v +(2\si+1)u_0^{2\si}\<x\>^k v = \<x\>^kg.
\end{equation*}
Introducing the commutator
\begin{equation*}
  [\<x\>^k,\Delta]= \<x\>^k\Delta -
\Delta\(\<x\>^k\cdot\)=-2k\<x\>^{k-2}x\cdot\nabla + \O\(\<x\>^{k-2}\),
\end{equation*}
the above equation also reads 
\begin{equation*}
  -\eps^2\Delta \( \<x\>^kv\) +(2\si+1)u_0^{2\si}\<x\>^k v = \<x\>^kg+
  \eps^2[\<x\>^k,\Delta]v.
\end{equation*}
Proceed like above, and take the $L^2$ inner product with $\<x\>^k v$:
on the one hand,
\begin{equation*}
  \< -\eps^2\Delta \( \<x\>^kv\) +(2\si+1)u_0^{2\si}\<x\>^k v, \<x\>^k
  v\> \ge \delta^{\frac{2\si}{2\si +1}}\|\<x\>^k  v\|_{L^2}^2,
\end{equation*}
and on the other hand,
\begin{align*}
  \< \<x\>^kg+
  \eps^2[\<x\>^k,\Delta]v, \<x\>^k  v\>
  &\le \|\<x\>^kg\|_{L^2} \|\<x\>^k
    v\|_{L^2} \\
  &\quad-2k  \eps^2\int_{\R^d} \<x\>^{k-2}x\cdot \nabla v \<x\>^k
    v \\
  &\quad+C \int_{\R^d} \<x\>^{k-2}|v|\<x\>^k|v|.
\end{align*}
The second term on the right hand side is integrated by parts,
\begin{align*}
  -2\int_{\R^d} \<x\>^{k-2}x\cdot \nabla v \<x\>^k
    v&= -\int_{\R^d} \<x\>^{2k-2}x\cdot \nabla \(v^2\) = \int_{\R^d}
       \nabla\cdot\(\<x\>^{2k-2}x\) v^2\\
  &=\O\(\int \<x\>^{2k-2}v^2\).  
\end{align*}
Cauchy-Schwarz inequality yields
\begin{equation*}
  \< \<x\>^kg+
  \eps^2[\<x\>^k,\Delta]v, \<x\>^k  v\>\le  \|\<x\>^k
    v\|_{L^2}\( \|\<x\>^kg\|_{L^2} + \eps^2 \|\<x\>^{k-2}v\|_{L^2}\).
  \end{equation*}
  Since we know from the first part of the lemma that
  \begin{equation*}
    \|v\|_{L^2}\lesssim \delta^{-\frac{2\si}{2\si+1}}\|g\|_{L^2}, 
  \end{equation*}
  the end of the lemma follows by induction. 
\end{proof}
We infer the following result, where we do not try to optimize the
parameters involved in the statement. 
\begin{proposition}\label{prop:exist-v}
  Let $k>d/2$ and $\si>0$  be integers. There exist $M(k,\si,d)>0$
  such that if
\begin{equation*}
  \eps\ll \delta^{2M}\ll 1,
\end{equation*}
then there exists $v\in H^k(\R^d)\cap \F
H^k(\R^d)$ such that $u_\eps=u_0+v$ solves
\begin{equation}\label{eq:u-eps-P0}
  -\eps^2\Delta u_\eps +|u_\eps|^{2\si}u_\eps = f,\quad f(x) =
  1-e^{-\delta}e^{-|x|^2/2}, 
\end{equation}
and
\begin{equation*}
  \|v\|_{H^k(\R^d)}  + \left\|\<\cdot \>^kv\right\|_{L^2(\R^d)}\lesssim
  \frac{\eps}{\delta^{M}}\ll \delta^M\ll 1. 
\end{equation*}
\end{proposition}
\begin{proof}
 Using the decomposition $u_\eps=u_0+v$, \eqref{eq:u-eps-P0} reads,
 in terms of $v$,
 \begin{equation*}
   -\eps^2\Delta v + (2\si+1)u_0^{2\si} v = -\eps^2\Delta u_0
   -2\si(2\si+1)v^2\int_0^1 (1-\theta)|u_0+\theta
   v|^{2\si-2}(u_0+\theta v)\dd\theta,
 \end{equation*}
 where we have used Taylor formula applied to the function
 $f(y)=|y|^{2\si}y$, for $y\in \R$. Denote by $G(\eps,u_0,v)$ the
above right hand side. We consider the iterative scheme
 \begin{equation}
 \label{iterative}
    -\eps^2\Delta v_{n+1} + (2\si+1)u_0^{2\si} v_{n+1} = G(\eps,u_0,v_n), 
 \end{equation}
 initialized with $v_0=0$. Lemmas~\ref{lem:u0Hs} and
 \ref{lem:LaxMilgram} yield, since  $H^k(\R^d)$
 is a Banach algebra and $\si$ is an integer,
 \begin{align*}
   \|v_{n+1}\|_{H^k}
   & \le
   \frac{C}{\delta^a}\(\eps^2 \delta^{\alpha- 1-k/2-d/4}+
   \|v_n\|_{H^k}^2\(\|u_0-1\|_{H^k}^{2\si-1} +
     \|v_n\|_{H^k}^{2\si-1}\)\)\\
& \le
   \frac{C}{\delta^a}\(\eps^2 \delta^{\alpha- 1-k/2-d/4}+
   \|v_n\|_{H^k}^2\(C\delta^{\alpha-k/2-d/4} +
     \|v_n\|_{H^k}^{2\si-1}\)\)
   . 
 \end{align*}
In the above estimate, we have considered $u_0-1$ and not $u_0$: as
$G(\eps,u_0,v_n)$ is a linear combination of terms of the form
$u_0^{2\si+1-\ell} v_n^{\ell}$, $\ell\ge 2$, either all the
derivatives fall on $v_n$, and we use the estimate
\begin{equation*}
  \|u_0^{2\si+1-\ell} \d^\beta v_n^{\ell}\|_{L^2}\le
  \|u_0\|_{L^\infty}^{2\si+1-\ell} \|\d^\beta v_n^{\ell}\|_{L^2}\le
  \|v_n\|_{H^k}^\ell,
\end{equation*}
or at least one derivative affects $u_0$, and then $u_0$ can be
replaced by $u_0-1$ in the corresponding term(s). 
We thus have an estimate of the form
\begin{equation*}
  \|v_{n+1}\|_{H^k}\le C\(\frac{\eps^2}{\delta^{\alpha_1}}+
  \delta^{-\alpha_2}\|v_n\|_{H^k}^2 + \|v_n\|_{H^k}^{2\si+1}\),\quad
  \alpha_1,\alpha_2>0. 
\end{equation*}
To simplify the presentation, we may suppose $\alpha_1=\alpha_2=:M$. We
see that the property $\|v_n\|_{H^k}\lesssim \eps  \delta^{-M}$ is
propagated by induction, provided that
\begin{equation*}
  \frac{\eps^2}{\delta^M}\ll \frac{\eps}{\delta^M},\quad
  \frac{\eps^2}{\delta^{3M}}\ll \frac{\eps}{\delta^M},\quad
  \frac{\eps}{\delta^M}\ll 1.
\end{equation*}
The strongest condition is the middle one, which yields the first
condition of the proposition. Up to decreasing the implicit constant
$0<\eta<1$ in the notation
\begin{equation*}
  \eps\ll \delta^{2M}\Longleftrightarrow \eps\le \eta\delta^{2M},
\end{equation*}
we infer that the series $\sum \|v_{n+1}-v_n\|_{H^k}$ converges
geometrically, that is,
\begin{equation*}
  v_n\Tend n \infty v\quad\text{in }H^k(\R^d),
\end{equation*}
and $u_\eps=u_0+v$ solves \eqref{eq:u-eps-P0}.
\smallbreak

Lemma~\ref{lem:LaxMilgram} also yields
\begin{equation*}
  \|\<\cdot\>^kv\|_{L^2}\lesssim \frac{1}{\delta^a}\(\|\eps^2\<\cdot\>^k\Delta
  u_0\|_{L^2}+ \|\<\cdot\>^k u_0^{2\si-1}v^2\|_{L^2}+ \|\<\cdot\>^k
  v^{2\si+1}\|_{L^2}\). 
\end{equation*}
In view of the estimate $\|u_0\|_{L^\infty}\le 1$ and since
$\|v\|_{L^\infty}\lesssim \|v\|_{H^k}\lesssim \eps\delta^{-M}$, we
infer, invoking Lemma~\ref{lem:u0Hs},
\begin{equation*}
  \|\<\cdot\>^kv\|_{L^2}\lesssim
  \frac{\eps^2}{\delta^{\alpha_3}} + \(\|v\|_{L^\infty} +
  \|v\|_{L^\infty}^{2\si}\) \|\<\cdot\>^kv\|_{L^2}\lesssim
  \frac{\eps^2}{\delta^{\alpha_3}} +
  \frac{\eps}{\delta^M}\|\<\cdot\>^kv\|_{L^2}, 
\end{equation*}
for some $\alpha_3>0$, hence, since $\frac{\eps}{\delta^M}\ll
\delta^M\ll 1$,
\begin{equation*}
  \|\<\cdot\>^kv\|_{L^2}\lesssim
  \frac{\eps^2}{\delta^{\alpha_3}},
\end{equation*}
and the proposition follows. 
\end{proof}
\begin{remark}[Case with damping]\label{rem:damping}
  If in \eqref{eq:NLS}, the term $\eps^2\Delta$ is replaced with
  $(1-i\nu)\eps^2\Delta$ with $\nu>0$ so the new term is dissipative,
  then we can adapt the above approach. 
  Indeed, we obtain in this case (with $u_0 \in \R$) 
\begin{equation*}
 i \nu \eps^2 \Delta v  -\eps^2\Delta v +(\si+1)|u_0|^{2\si}v + \si
  |u_0|^{2\si} \bar v = -(1 - i \nu) \eps^2\Delta
  u_0+\O\(|v|^2\). 
\end{equation*}
Taking the real and imaginary parts by writing $v = p + iq$ linear operator associated with this system is 
\begin{align*}
\begin{pmatrix}
- \eps^2 \Delta + (2\si+1)u_0^{2\si} & - \nu \eps^2 \Delta \\
\nu \eps^2 \Delta  & - \eps^2 \Delta + u_0^{2\si}
\end{pmatrix}\begin{pmatrix}
p \\q 
\end{pmatrix}
\end{align*}
We can write this system as 
\begin{align*}
\begin{pmatrix}
A & C \\
- C^T & B
\end{pmatrix}\begin{pmatrix}
p \\q 
\end{pmatrix} = \begin{pmatrix}
f \\ g 
\end{pmatrix}
\end{align*}
where $A$ and $B$ are similar to the previous case (invertible in $H^s$ with inverse norm independent of $\eps$). 
When we solve the linear system we can write: 
\begin{equation*}
p = A^{-1} f -  A^{-1}C q,  
\end{equation*}
and thus
\begin{equation*}
  (B + C^T A^{-1} C ) q = C^T A^{-1} f + g,
\end{equation*}
and $B + C^T A^{-1} C$, which is now self-adjoint should be invertible
for $\nu$ small enough.  
Indeed, the estimate \eqref{coercivity} when $\varepsilon^2 \ll \delta^{2\si}$ implies that 
\begin{equation*}
\|A^{-1} C  u \|_{\dot H^s} \le C \nu \|u \|_{H^s},
\end{equation*}
uniformly in $\varepsilon$. Hence we have
\begin{equation*}
  \|C A^{-1} C  u \|_{H^s} \le C \varepsilon^2  \nu \|u \|_{H^{s+2}}, 
\end{equation*}
and in particular for $\nu$ small enough, the operator 
$B + C^T A^{-1} C$ satisfies the coercive estimate \eqref{coercivity}
as well as the uniform estimate \eqref{ellipunif}. We thus deduce that
the same result holds for the dissipative case. Cumbersome details are
left out. 
\end{remark}

\subsubsection*{Time dependent forcing and cubic nonlinearity}

In the case where $f=2Q(x)e^{-3it}$ and $\si=1$ considered in Section~\ref{sec:time-periodic} (see
Theorem~\ref{theo:cubic}), the equation for $v$ becomes  
 \begin{equation*}
   -\eps^2\Delta v -3Pv+3U_k^{2}v = -\eps^2\Delta
  U_k+\O\(v^2\).
\end{equation*}
Recall that for $k=0,1,2$, $U_k$ is given by
\begin{equation*}
  U_k(x) =2\cos\( \frac{\theta(x)}{3}+\frac{2k\pi}{3}\),\quad \cos \(
  \frac{\theta(x)}{3}\) =e^{-\frac{\delta}{2}-\frac{|x|^2}{4}}\in
  \left(0,e^{-\frac{\delta}{2}}\right].
\end{equation*}
In particular,
\begin{equation*}
  \frac{\theta(x)}{3}\in
  \left[\arccos e^{-\frac{\delta}{2}},\frac{\pi}{2}\right),
\end{equation*}
and therefore
\begin{equation*}
  \frac{\theta(x)}{3}+\frac{2k\pi}{3}\in \left[\arccos
    e^{-\frac{\delta}{2}}+\frac{2k\pi}{3}, (4k+3)\frac{\pi}{6}\right). 
\end{equation*}
This shows that $|U_0|$ and $|U_2|$ are arbitrarily close to zero. On
the other hand, writing for instance
\begin{equation*}
  U_1(x)
   = - \cos\(\frac{\theta(x)}{3}\)- \sqrt 3 \sin
    \(\frac{\theta(x)}{3}\)
   =-e^{-\frac{\delta}{2}-\frac{|x|^2}{4}} - \sqrt 3
    \(1-e^{-\delta-\frac{|x|^2}{2}}\)^{1/2},
\end{equation*}
we compute
\begin{equation*}
  |U_1(x)|^2 \ge e^{-\delta-\frac{|x|^2}{2}}
  +3\(1-e^{-\delta-\frac{|x|^2}{2}}\)= 3-2 e^{-\delta-\frac{|x|^2}{2}}\ge3-2e^{-\delta}>1,
\end{equation*}
and we see that the operator $-\eps^2\Delta  -3+3U_1^{2}$ is uniformly
elliptic  in the sense that
\begin{equation*}
  \<-\eps^2\Delta v -3v+3U_1^{2}v,v\>\ge \eps^2\|\nabla v\|_{L^2}^2 +
  c(\delta)\|v\|_{L^2}^2,
\end{equation*}
where $c(\delta)>0$ provided that $\delta>0$.
We can thus resume the proof of Lemma~\ref{lem:LaxMilgram}. 
When considering derivatives of $v$, we have like above
 \begin{equation*}
   -\eps^2\Delta (D v) -v+3U_1^{2}Dv = Dg-6U_1DU_1 v,
\end{equation*}
and we can conclude similarly.

\subsubsection*{End of the proof of Theorem~\ref{theo:constr-stat}}
The idea now is to use Sobolev embedding at the level of $\hat
u_\eps$, as opposed to $u_\eps$ directly, and write,
\begin{equation*}
 \sup_{\xi\in \R^d} \left||\xi|^{d+2\alpha}\(\widehat u_\eps(\xi)-\widehat
    u_0(\xi)\)\right| \lesssim \left\||\cdot|^{d+2\alpha}
    \(\widehat u_\eps-\widehat u_0\)\right\|_{H^s(\R^d)}, 
\end{equation*}
provided that $s>d/2$, where $\alpha$ corresponds to the frequency
cascade for $u_0$, as stated in Theorem~\ref{theo:constr-stat}. In
view of Plancherel formula, 
\begin{equation*}
  \left\||\cdot|^{d+2\alpha}\(\widehat u_\eps-\widehat
    u_0\)\right\|_{H^s(\R^d)}= \left\|\<\cdot\>^s \Lambda^{d+2\alpha}\( u_\eps-
    u_0\)\right\|_{L^2(\R^d)},\quad \Lambda=(-\Delta)^{1/2}. 
\end{equation*}
The operator involved above has symbol $\<x\>^s |\xi|^{{d+2\alpha}}$,
and Young inequality yields the pointwise estimate
\begin{equation*}
  \<x\>^s |\xi|^{d+2\alpha} \le \<x\>^{3s}+
  |\xi|^{\frac{3}{2}(d+2\alpha)}. 
\end{equation*}
Let
\begin{equation*}
  k = \left\lceil \frac{3}{2}(d+2\alpha)\right\rceil. 
\end{equation*}
Then $k>3d/2$, and we set $s=k/3>d/2$, so we have the pointwise estimate
\begin{equation*}
  \<x\>^s |\xi|^{d+2\alpha} \le \<x\>^{k}+
  \<\xi\>^k. 
\end{equation*}
Formally, this implies
\begin{equation*}
  \left\|\<\cdot\>^s \Lambda^{d+2\alpha}\( u_\eps-
    u_0\)\right\|_{L^2(\R^d)}\lesssim \left\|\<\cdot\>^k\( u_\eps-
    u_0\)\right\|_{L^2(\R^d)}+ \|u_\eps-
    u_0\|_{H^k(\R^d)}.
\end{equation*}
The justification of this inequality is not completely obvious though,
and we refer to 
\cite{Helffer1984} (see also \cite{BCM08}) for a complete proof, based
on pseudodifferential calculus.

\smallbreak

Invoking Proposition~\ref{prop:exist-v} with the above $k$, we get
some $M>0$ such that if $\eps\ll \delta^{2M}\ll 1$,
\begin{equation*}
  \left\|\<\cdot\>^k\( u_\eps-
    u_0\)\right\|_{L^2(\R^d)}+ \|u_\eps-
    u_0\|_{H^k(\R^d)}\lesssim \frac{\eps}{\delta^M}\ll
  \delta^M\ll 1,
\end{equation*}
and thus
\begin{equation*}
  \sup_{\xi\in \R^d} \left||\xi|^{d+2\alpha}\(\widehat u_\eps(\xi)-\widehat
    u_0(\xi)\)\right| \lesssim \frac{\eps}{\delta^M}\ll
  \delta^M\ll 1,
\end{equation*}
so Theorem~\ref{theo:constr-stat} follows with $K=M/2$. 
\subsection{Conserved energy}
\label{sec:energy}

Even in the presence of the forcing term $f$ (which we keep
considering independent of time now), the equation \eqref{eq:NLS}
enjoys a Hamiltonian structure. To see that some energy is conserved,
one may follow the usual strategy for Schr\"odinger equations (see
e.g. \cite[Section~3.1]{CazCourant}): multiply \eqref{eq:NLS} by $\d_t
\overline u$, 
integrate in space, and consider the real part of the outcome. This
yields formally
\begin{equation*}
  \frac{\dd}{\dd t}E_\eps\(u(t)\)=0,
\end{equation*}
where, since $f$ does not depend on time, the energy $E_\eps$ is
defined by 
\begin{equation*}
  E_\eps(w) = \frac{\eps^2}{2}\|\nabla w\|_{L^2(\R^d)}^2
  +\frac{1}{2\si+2}\|w\|_{L^{2\si+2}(\R^d)}^{2\si+2} -\RE
  \int_{\R^d}f\overline w.
\end{equation*}
As a matter of fact, this computation is only formal in the framework
that we consider, since as pointed out before, we consider
perturbations of the algebraic solution $u_0$, and $u_0$ belongs to no
other Lebesgue space than $L^\infty(\R^d)$. Writing $u=u_0+v$, giving
a rigorous meaning to a conserved energy for $v$ becomes possible, at
least when $\si$ is an integer, in
the spirit of renormalization. For simplicity, we consider the case
$\si=1$. 

Recall the standard notation
\begin{equation*}
  \< g,h \> = \mathrm{Re}\int_{\R^d} g \bar h.
\end{equation*}
\begin{lemma}
Let $d\ge 1$, and assume that $f$ is real-valued. 
Let $u_\eps(x) \in \R$ be such that
\begin{equation*}
  - \eps^2 \Delta u_\eps   +
  u_\varepsilon^{3}=f(x). 
\end{equation*}
Let  $u$ be the solution to \eqref{eq:NLS} with $\si=1$, and let $v =
u - u_\eps$.  If $u$ is such that $v\in C([0,T];H^1\cap L^4)$, 
then formally,
\begin{equation*}
 \E(v)  = \frac{\eps^2}{2}\|\nabla v\|_{L^2}^2+  \| u_\eps
 v\|_{L^2}^2 + 
 \int u_\eps^2\(\RE v\)^2
+  \RE\int |v|^2 v   u_\eps
 + \frac14 \|v\|_{L^4}^4
\end{equation*}
is independent of $t\in [0,T]$. 
\end{lemma}
\begin{proof}
  We first note the identity, for $w_1,w_2\in H^1\cap L^4$ independent
  of time,
  \begin{equation*}
    E_\eps(w_1+w_2) = E_\eps(w_1)+\RE\int \overline{w_2}\( -\eps^2\Delta w_1+|w_1|^2w_1-f\)+\E(w_2),
  \end{equation*}
where, in the definition of $\E$, we have replaced $u_\eps$ with
$w_1$. Substituting $w_1$ with $u_\eps$ and $w_2$ with $v$ formally
yields the conservation of $\E(v)$. 
Proceeding like in \cite[Lemma~6.5.3]{CazCourant}, that is,
multiplying $u_\eps$ and $v$ by $e^{-\eta|x|^2}$ then letting $\eta\to
0$, we infer
that if $v$ has the regularity assumed in the statement, then indeed
$\E(v)$ is independent of time. 
\end{proof}
Cauchy-Schwarz inequality implies, since $u_\eps$ is real-valued,
\begin{equation*}
  \left|\int u_\eps|v|^2 \RE v   \right|\le \(\int |v|^4\)^{1/2}\(
  \int u_\eps^2(\RE v)^2\)^{1/2},
\end{equation*}
hence by Young inequality,
\begin{equation*}
  \left|\RE\int |v|^2 v  \bar u_\eps\right|\le \(\int |v|^4\)^{1/2}\(
  \int u_\eps^2(\RE v)^2\)^{1/2}\le \frac{1}{4}\int |v|^4+ \int
  u_\eps^2(\RE v)^2. 
\end{equation*}
This readily yields
\begin{equation*}
  \E(v)  \ge \frac{\eps^2}{2}\|\nabla v\|_{L^2}^2+ 
 \int u_\eps^2|v|^2,
\end{equation*}
which is a rather weak form of coercivity, since the lower bound for
$u_\eps^2$ involves a positive power of $\delta$,
as
\begin{equation*}
  u_\eps(x)\ge u_0(x)-\|v\|_{\infty}\ge
  \(1-e^{-\delta}\)^{\frac{1}{2\si+1}}- \delta^M\gtrsim
  \delta^{\frac{1}{2\si+1}}, 
  \end{equation*}
  where we have used Proposition~\ref{prop:exist-v} and Sobolev
  embedding to neglect the
  contribution of $v$. 
It is possible to
infer some dynamical stability from such estimates, up to losing some
powers of $\eps$ or $\delta$. Loosely speaking, if $\E(v)$ is small
initially, then $\eps\|\nabla v(t)\|_{L^2}$ and
$\delta\|v(t)\|_{L^2}$ remain small as long as $v\in C([0,T];H^1\cap
L^4)$, which yields some smallness for $v$ in $H^1$, provided that $\E(v_{\mid
  t=0})$ is sufficiently small in terms of $\eps$ and $\delta$.  
We do not pursue this direction here, which could easily lead to
cumbersome statements.

\subsection{Stability via Strichartz estimates}

 When the nonlinearity is
smooth and $L^2$-subcritical or critical, we can directly get a
dynamical stability result thanks to Strichartz estimates. The cases
thus covered are:
\begin{itemize}
\item cubic or quintic in 1D ($d=1$ and $\si=1$ or $2$),
\item cubic in 2D ($d=2$ and $\si=1$). 
\end{itemize}

We consider perturbations of the form $u=\u_0+v$: the goal is to prove
that $v$ is small in $L^2$ on some nontrivial time interval. 
The nonlinearity has to be smooth because $\u_0$ is in no Lebesgue
space apart from $L^\infty$, so for instance $|\u_0|^2v$ and $|v|^2v$
are not estimated in the same space, so writing
\begin{equation*}
  \left| |\u_0+v|^{2\si}(\u_0+v)-|\u_0|^{2\si}\u_0\right|\lesssim
  \(|\u_0|^{2\si}+|v|^{2\si}\)|v| 
\end{equation*}
is not suited to the use of Strichartz estimates in this context. The
nonlinearity has to be smooth so
$|\u_0+v|^{2\si}(\u_0+v)-|\u_0|^{2\si}\u_0$ can be written as a finite sum
of terms like evoked above.

\begin{definition}\label{def:adm}
Let $d=1$ or $2$. 
 A pair $(q,r)$ is {\bf admissible} if $2\le
 r\le\infty$ if $d=1$, $2\le r< 
  \infty$ if $d=2$,
  and 
  \begin{equation*}
    \frac{2}{q}=\delta(r):= d\left( \frac{1}{2}-\frac{1}{r}\right).
  \end{equation*}
\end{definition}
The following Strichartz estimates easily follow from the usual ones
(see e.g. \cite{CazCourant,TaoDisp}) by considering the rescaling
$u(t,\frac{x}{\eps})$: 
\begin{lemma}[Strichartz
  estimates]\label{lem:strichartz}  
Denote $U^\eps_0(t)=e^{i\eps^2t \Delta}$.\\
$(1)$ \emph{Homogeneous Strichartz estimate.} For any admissible pair
\index{admissible pair} $(q,r)$,
there exists $C_{q}$ independent of $\eps$ such that
\begin{equation*}
  \eps^{2/q}\|U^\eps_0 \varphi\|_{L^{q}(\R;L^{r})} \le C_q 
\|\varphi \|_{L^2},\quad \forall \varphi\in L^2(\R^d).
\end{equation*}
$(2)$ \emph{Inhomogeneous Strichartz estimate.}
For a time interval $I$, denote
\begin{equation*}
  D_I^\eps(F)(t,x) = \int_{I\cap\{\tau\le
      t\}} U^\eps_0(t-\tau)F(\tau,x)\dd \tau. 
\end{equation*}
For all admissible pairs \index{admissible pair}$(q_1,r_1)$ and~$
    (q_2,r_2)$, and any 
    interval $I$, there exists $C=C_{r_1,r_2}$ independent of $\eps$ and
    $I$ such that 
\begin{equation}\label{eq:strichnl}
      \eps^{2/q_1+2/q_2}\left\lVert D_I^\eps(F)
      \right\rVert_{L^{q_1}(I;L^{r_1})}\le C \left\lVert
      F\right\rVert_{L^{q'_2}\(I;L^{r'_2}\)},
\end{equation}
for all $F\in L^{q'_2}(I;L^{r'_2})$.
\end{lemma}
\noindent {\bf First case: $\si=1$, with $d=1$ or $d=2$.}
A function
$u=\u_0+v$ solves \eqref{eq:NLS} if and only if $v$ solves
\begin{equation}
  \label{eq:v}
  i\d_t v +\eps^2\Delta v = 2|\u_0|^2 v + \u_0^2\bar v + 2\u_0|v|^2+\bar
  \u_0 v^2+|v|^2v-\eps^2\Delta \u_0. 
\end{equation}
With initial data $v_0$ and $t\in I\subset [0,+\infty)$, Duhamel's formula reads
\begin{equation*}
  v(t) = U_0^\eps(t)v_0 -i D^\eps_I \( 2|\u_0|^2 v + \u_0^2\bar v + 2\u_0|v|^2+\bar
  \u_0 v^2+|v|^2v-\eps^2\Delta \u_0\). 
\end{equation*}
Local existence in the type of space evoked in
Theorem~\ref{theo:stab-strichartz} follows from a standard fixed point
argument on the above Duhamel's formula, involving the same estimates
as below. We focus on such estimates and the bootstrap argument, and
leave out the local existence part.

Let $q_0$ such that the pair $(q_0,4)$ is admissible:
$q_0=\frac{8}{d}$, so $q_0=8$ if
$d=1$, $q_0=4$ if $d=2$. Strichartz estimates yield
\begin{align*}
  \|v\|_{L^\infty(I;L^2)}&\lesssim \|v_0\|_{L^2} + \| |\u_0|^2
  v\|_{L^1(I;L^2)} + \|\u_0 v^2\|_{L^{1}(I;L^2)}\\
&\quad +\eps^{-d/4} \|
  v^3\|_{L^{q_0'}(I;L^{4/3})}  + \|\eps^2\Delta \u_0\|_{L^{1}(I;L^2)}. 
\end{align*}
Using H\"older inequality, and the fact that $\|\u_0\|_{L^\infty}\le 1$,
\begin{align*}
  \|v\|_{L^\infty(I;L^2)}&\lesssim \|v_0\|_{L^2} + \| 
  v\|_{L^1(I;L^2)} + \| v\|^2_{L^{2}(I;L^4)}\\
&\quad +\eps^{-d/4} \|
  v\|^3_{L^{3q_0'}(I;L^{4})}  + |I| \|\eps^2\Delta \u_0\|_{L^2}. 
\end{align*}
As $q_0>2$ and $3q_0'\le q_0$ (with equality in the $L^2$-critical case $d=2$), 
\begin{equation*}
  \| v\|_{L^{2}(I;L^4)}\le |I|^{a} \| v\|_{L^{q_0}(I;L^4)},\quad  \|
  v\|_{L^{3q_0'}(I;L^{4})}  \le |I|^{b}\|
  v\|_{L^{q_0}(I;L^{4})}  ,
\end{equation*}
for some $a>0$ and $b\ge 0$, and so, for $I\subset [0,T]$ with $T$ finite,
\begin{equation*}
   \|v\|_{L^\infty(I;L^2)}\lesssim \|v_0\|_{L^2} + \| 
  v\|_{L^1(I;L^2)} + \| v\|^2_{L^{q_0}(I;L^4)}+\eps^{-d/4} \|
  v\|^3_{L^{q_0}(I;L^{4})}  +  \|\eps^2\Delta \u_0\|_{L^2}.
\end{equation*}
We get rid of the second term on the right hand side by Gr\"onwall
lemma: for $I=[0,t]$, $0\le t\le T$,
\begin{equation}\label{eq:errL2}
  \begin{aligned}
   \|v\|_{L^\infty(I;L^2)}& \lesssim e^{Ct}\|v_0\|_{L^2} + e^{Ct}\|
   v\|^2_{L^{q_0}(I;L^4)}+\eps^{-d/4} e^{Ct}\| 
                            v\|^3_{L^{q_0}(I;L^{4})}\\
    &\quad+  e^{Ct}\|\eps^2\Delta \u_0\|_{L^2}.
    \end{aligned}
  \end{equation}
  We note that the constant $C$ can essentially be taken equal to one
  (slightly above),
  since there is no constant in the $L^2$-estimate, and we have
  absorbed algebraic terms $|I|^a=t^a$ in the exponential.
We then use Strichartz estimates again, but with the $L^{q_0}L^4$-norm
on the left hand side:
\begin{align*}
  \|v\|_{L^{q_0}(I;L^4)}&\lesssim \eps^{-d/4}\|v_0\|_{L^2} +
                          \eps^{-d/4}\| |\u_0|^2 
  v\|_{L^1(I;L^2)} + \eps^{-d/4}\|\u_0 v^2\|_{L^{1}(I;L^2)}\\
&\quad +\eps^{-d/2} \|
  v^3\|_{L^{q_0'}(I;L^{4/3})}  + \eps^{-d/4}\|\eps^2\Delta \u_0\|_{L^{1}(I;L^2)}\\
&\lesssim \eps^{-d/4}\|v_0\|_{L^2} + \eps^{-d/4}\| 
  v\|_{L^1(I;L^2)} + \eps^{-d/4}\|v\|^2_{L^{2}(I;L^4)}\\
&\quad +\eps^{-d/2} \|
  v\|^3_{L^{3q_0'}(I;L^{4})}  + \eps^{-d/4} T \|\eps^2\Delta
                                                           \u_0\|_{L^2}\\
&\lesssim \eps^{-d/4}\|v_0\|_{L^2} +
    \eps^{-d/4}\|v\|^2_{L^{q_0}(I;L^4)} +\eps^{-d/2} \| 
  v\|^3_{L^{q_0}(I;L^{4})}  + 
\eps^{-d/4} T \|\eps^2\Delta   \u_0\|_{L^2}\\
+ \eps^{-d/4}&\( e^{Ct}\|v_0\|_{L^2} + e^{Ct}\|
   v\|^2_{L^{q_0}(I;L^4)}+\eps^{-d/4} e^{Ct}\| 
  v\|^3_{L^{q_0}(I;L^{4})}  +  e^{Ct}\|\eps^2\Delta \u_0\|_{L^2}\),
\end{align*}
where we have estimated $\|v\|_{L^1(I;L^2)}$ by the first step. It
yields the largest contribution, and thus
\begin{equation}\label{eq:L4}
\begin{aligned}
   \|v\|_{L^{q_0}(I;L^4)} & \lesssim  \eps^{-d/4} e^{Ct}\|v_0\|_{L^2}
                            + \eps^{-d/4}  e^{Ct}\| 
   v\|^2_{L^{q_0}(I;L^4)}\\
&\quad+\eps^{-d/2} e^{Ct}\| 
  v\|^3_{L^{q_0}(I;L^{4})}  +  \eps^{-d/4}  e^{Ct}\|\eps^2\Delta \u_0\|_{L^2}.
\end{aligned}
\end{equation}
We argue by bootstrap. So long as
\begin{equation*}
  \|v\|_{L^{q_0}(0,t;L^4)}\le \eta \eps^{d/4},
\end{equation*}
for some  $\eta>0$ sufficiently small,
the terms  $\|v\|_{L^{q_0}(0,t;L^4)}$ on the right hand side of
\eqref{eq:L4} are absorbed by the left hand side.

One can take $\eta= \eps^\iota$ with $\iota>0$, and an  Ehrenfest time
is allowed, like $c\log\frac{1}{\eps}$ (until the exponential becomes
too big). More precisely:
\begin{align*}
  &\|v_0\|_{L^2}\le \eta \eps^{d/2}\quad \text{or}\quad \eps^{d/2+\iota},\\
  & \eps^2\|\Delta \u_0\|_{L^2}\le \eta \eps^{d/2}\quad \text{or}\quad
    \eps^{d/2+\iota}. 
\end{align*}
In view of Lemma~\ref{lem:u0Hs}, the last condition follows from
\begin{equation*}
  \eps^{2-d/2} \ll \delta^{1+d/4-\alpha}, 
\end{equation*}
as $1+d/4>\alpha$ under the assumptions of
Theorem~\ref{theo:stab-strichartz}. 
\begin{remark}
  If $\alpha>1+d/4$
(see Remark~\ref{rem:other}), no extra smallness assumption is
needed. 
\end{remark}
Once the bootstrap condition is granted, we resume the estimate
\eqref{eq:errL2},
\begin{equation*}
  \|v\|_{L^\infty(0,t;L^2)}\lesssim  \eps^{d/2}  e^{Ct}.
\end{equation*}
\noindent {\bf Second case: $\si=2$ and $d=1$.} We have
\begin{align*}
  \|v\|_{L^qL^r}&\lesssim \eps^{-2/q}\|v_0\|_{L^2} + \sum_{j=1}^5
    \eps^{-2/q-2/q_j} \|\u_0^{5-j}v^j\|_{L^{q_j'}L^{r_j'}} +
    \eps^{-2/q}|I| \|\eps^2\Delta \u_0\|_{L^2}\\
&\lesssim \eps^{-2/q}\|v_0\|_{L^2} + \sum_{j=1}^5
    \eps^{-2/q-2/q_j} \|v^j\|_{L^{q_j'}L^{r_j'}} +
    \eps^{-2/q}|I| \|\eps^2\Delta \u_0\|_{L^2},
\end{align*}
where all the above pairs are admissible. We now choose successively
$(q,r)=(\infty,2)$ and $(q,r)=(6,6)$. 
\begin{itemize}
\item For $j=1$, we choose $(q_1,r_1) = (\infty,2)$.
\item For $j=2$, we choose $(q_2,r_2) = (12,3)$, and use H\"older
  inequality
  \begin{equation*}
    \|v^2\|_{L^{12/11}L^{3/2}}\le \|v\|_{L^{4/3}L^2}\|v\|_{L^6L^6}.
  \end{equation*}
\item For $j=3$, we choose $(q_3,r_3)= (\infty,2)$, and write
  \begin{equation*}
    \|v^3\|_{L^1L^2} =\|v\|_{L^3L^6}^3\le |I|^{1/2} \|v\|^3_{L^6L^6}.
  \end{equation*}
\item For $j=4$, we choose $(q_4,r_4)=(12,3)$, and use H\"older
  inequality
  \begin{equation*}
    \|v^4\|_{L^{12/11}L^{3/2}}=\|v\|^4_{L^{11/3}L^6}\le
    |I|^{7/66}\|v\|^4_{L^6L^6}. 
  \end{equation*}
\item For $j=5$, we choose $(q_5,r_5)=(6,6)$, and write
  \begin{equation*}
    \|v^5\|_{L^{6/5}L^{6/5}} = \|v\|^5_{L^6L^6}. 
  \end{equation*}
\end{itemize}
We come up with:
\begin{align*}
  \|v\|_{L^\infty L^2}&\lesssim \|v_0\|_{L^2} +  \|v\|_{L^1 L^2} +
  \eps^{-1/6}\|v\|_{L^6L^6} |I|^{3/4}\|v\|_{L^\infty L^2}+ |I|^{1/2}
                        \|v\|^3_{L^6L^6} \\
&\quad+\eps^{-1/6}  |I|^{7/66}\|v\|^4_{L^6L^6}+
         \eps^{-1/3}\|v\|^5_{L^6L^6}+|I|\|\eps^2\Delta \u_0\|_{L^2}. 
\end{align*}
Gr\"onwall lemma implies
\begin{align*}
  \|v\|_{L^\infty L^2}&\lesssim e^{Ct}\|v_0\|_{L^2} +
  \eps^{-1/6}\|v\|_{L^6L^6}e^{Ct} \|v\|_{L^\infty L^2}+ e^{Ct}
                        \|v\|^3_{L^6L^6} \\
&\quad+\eps^{-1/6} e^{Ct} \|v\|^4_{L^6L^6}+
         \eps^{-1/3}e^{Ct}\|v\|^5_{L^6L^6}+e^{Ct}\|\eps^2\Delta \u_0\|_{L^2}. 
\end{align*}
The bootstrap condition is now
\begin{equation}\label{eq:bootstrap-quintic}
   \eps^{-1/6}e^{Ct}\|v\|_{L^6L^6}\ll 1 ,
\end{equation}
which implies that in the above estimate, the second term on the right hand side
is absorbed by the left hand side,
\begin{align*}
  \|v\|_{L^\infty L^2}&\lesssim e^{Ct}\|v_0\|_{L^2} + e^{Ct}
                        \|v\|^3_{L^6L^6} +\eps^{-1/6} e^{Ct} \|v\|^4_{L^6L^6}\\
&\quad+
         \eps^{-1/3}e^{Ct}\|v\|^5_{L^6L^6}+e^{Ct}\|\eps^2\Delta \u_0\|_{L^2}. 
\end{align*}
Applying Strichartz inequality again, changing only $(q,r)$ to consider
$(q,r)=(6,6)$, we infer
\begin{align*}
  \|v\|_{L^6 L^6}&\lesssim \eps^{-1/3}e^{Ct}\|v_0\|_{L^2} 
 +\eps^{-1/3} e^{Ct}
   \|v\|^3_{L^6L^6} +\eps^{-1/2} e^{Ct} \|v\|^4_{L^6L^6}\\
&\quad+
         \eps^{-2/3}e^{Ct}\|v\|^5_{L^6L^6}+\eps^{-1/3}e^{Ct}\|\eps^2\Delta
                                              \u_0\|_{L^2}.  
\end{align*}
Under the bootstrap assumption \eqref{eq:bootstrap-quintic}, the terms
involving $\|v\|_{L^6L^6}$ on the right hand side are absorbed by the
left hand side, and 
\begin{equation*}
   \|v\|_{L^6 L^6}\lesssim \eps^{-1/3}e^{Ct}\|v_0\|_{L^2}
   +\eps^{-1/3}e^{Ct}\|\eps^2\Delta    \u_0\|_{L^2}.
\end{equation*}
Therefore, the assumption on $v_0$ and $\u_0$ for the argument to be
conclusive is, like in the one-dimensional cubic case,
\begin{equation*}
  \|v_0\|_{L^2} + \|\eps^2\Delta    \u_0\|_{L^2}\ll \sqrt\eps. 
\end{equation*}
Invoking Lemma~\ref{lem:u0Hs} again, the second condition is satisfied
provided that
\begin{equation*}
  \eps^2\delta^{\frac{1}{5}-1-\frac{1}{4}}\ll
  \sqrt\eps\Longleftrightarrow
  \eps^{2-\frac{d}{2}}\ll
  \delta^{\frac{d}{4}+1-\frac{1}{2\si+1}},\quad \text{as }d=1 \text{
    and }\si=2, 
\end{equation*}
and we conclude like in the first case $\si=1$, thus proving
Theorem~\ref{theo:stab-strichartz}.

\section{Numerical simulations}
\label{sec:num}

We conclude this paper by exhibiting numerically some solutions constructed above. 
We consider the case $d = 1$ and the cubic nonlinearity. 
All numerical simulations are done using $N = 2^{14}$ grid points on a
interval $[-2\pi,2\pi]$. In Figure~\ref{fig1} we illustrate
Theorem~\ref{theo:law-f} by plotting the solution $\widehat
u_0(\xi)$ obtained by computing $f(x)^{\frac13}$ and taking the
discrete Fourier transform. We observe the expected power law
$|\xi|^{-\frac53}$ on an interval of larger size when $\delta$ gets smaller.
\begin{figure}[ht]
\begin{center}
   \rotatebox{0}{\resizebox{!}{0.5\linewidth}{%
   \includegraphics{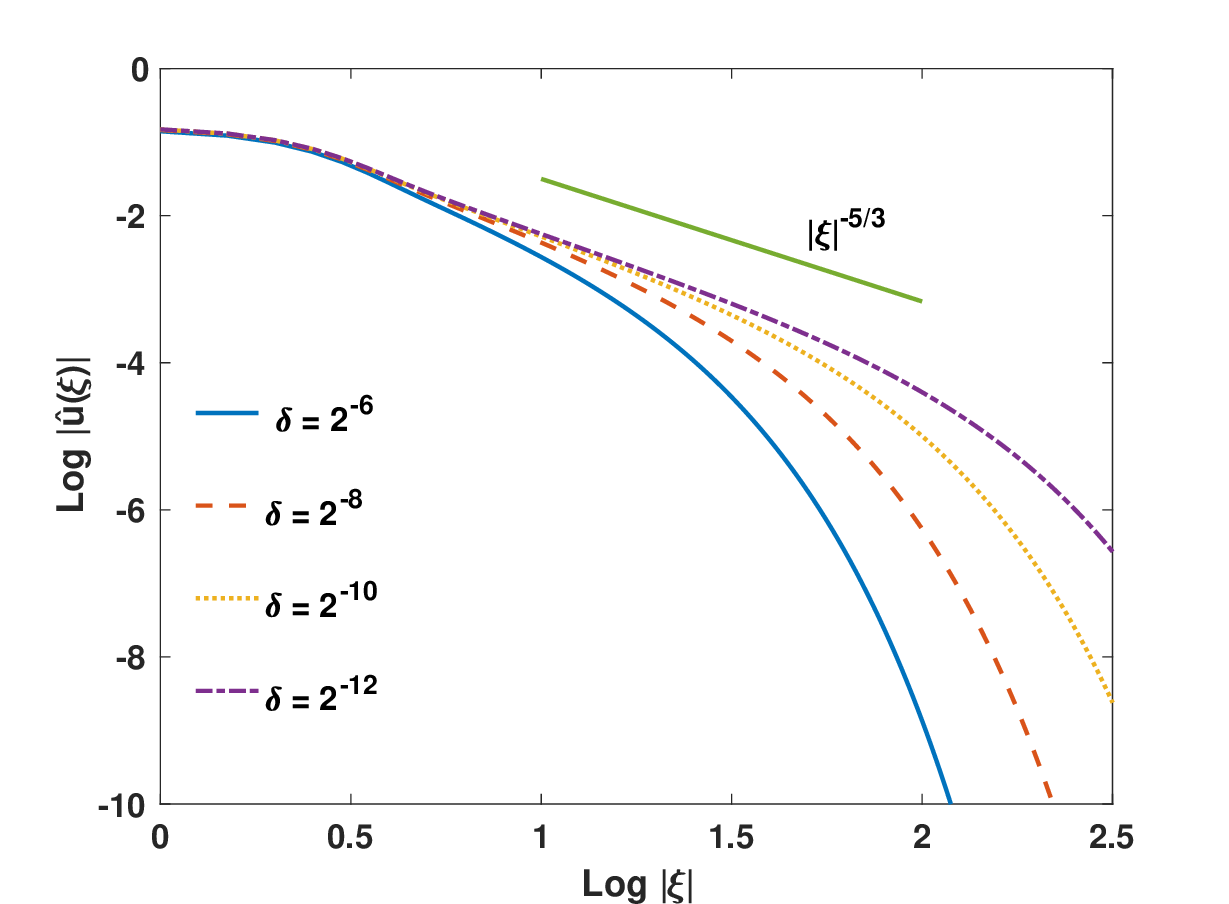}}}
\end{center}
\caption{$\hat u_0(\xi)$ in the case $P= 0$, $d = 1$, $\varepsilon = 0$}
\label{fig1}
\end{figure}

In Figure~\ref{fig2}, we fix $\delta$ and compute the solutions
$\widehat u_\varepsilon(\xi)$ given by Theorem~\ref{theo:constr-stat}
for various values of $\varepsilon$. We observe that when $\varepsilon$
becomes small enough, the solution indeed converges towards $u_0$. The
numerical scheme used to calculate $u_\varepsilon$ is the
implementation of the iterative scheme \eqref{iterative}. The
discretization of the Laplace operator is done by standard finite
differences with Dirichlet boundary conditions, which allows a simple
inversion of the linearized operator. Note that when $\varepsilon$
becomes too large, the solution becomes sensitive to the numerical
discretization and the existence of solution is unclear.  

\begin{figure}[ht]
\begin{center}
   \rotatebox{0}{\resizebox{!}{0.5\linewidth}{%
   \includegraphics{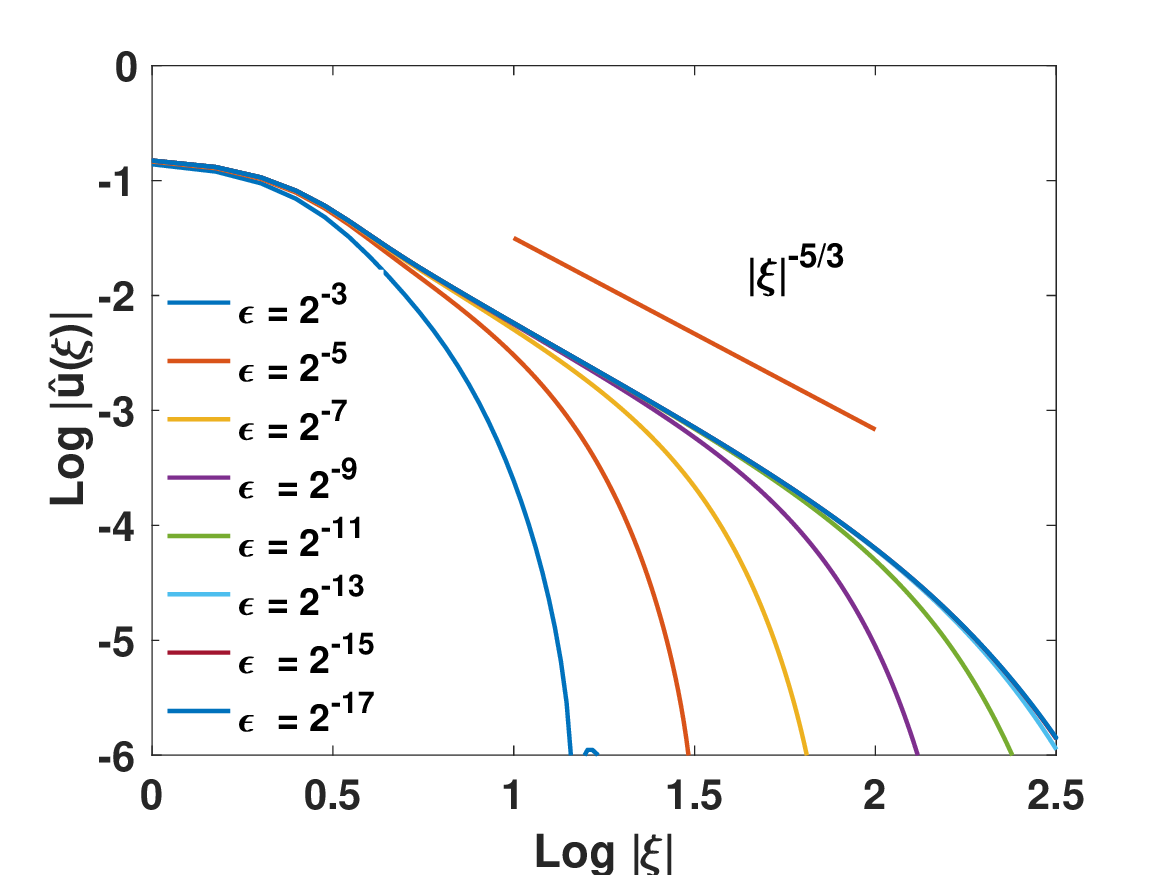}}}
\end{center}
\caption{$\hat u_0(\xi)$ in the case $P= 0$, $d = 1$, $\delta = 10^{-4}$}
\label{fig2}
\end{figure}

In Figures~\ref{fig3} and \ref{fig4} we repeat the same computations,
in the case $P = 1$ given by Theorems~\ref{theo:cubic} and
\ref{theo:constr-stat}, confirming the expected power law
$|\xi|^{-2}$.  

\begin{figure}[ht]
\begin{center}
   \rotatebox{0}{\resizebox{!}{0.5\linewidth}{%
   \includegraphics{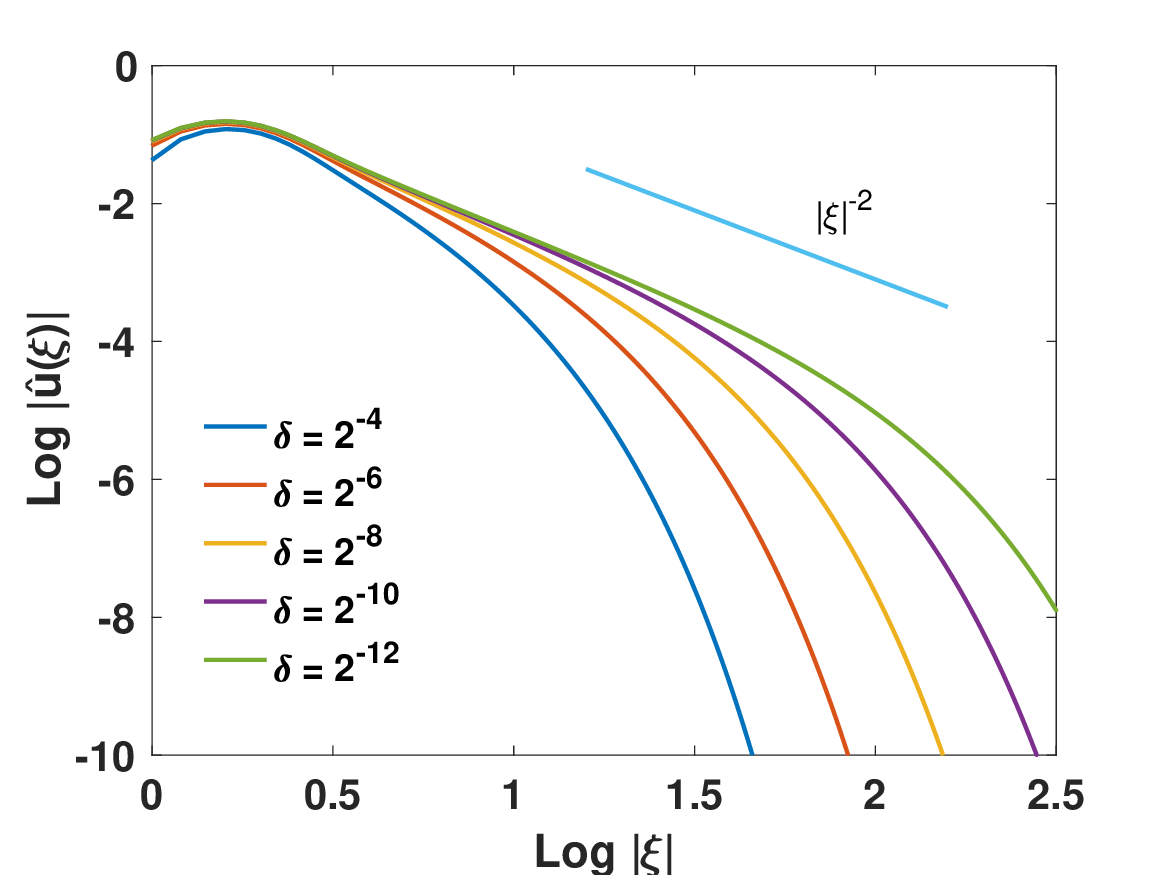}}}
\end{center}
\caption{$\hat u_0(\xi)$ in the case $P = 1$, $d = 1$, $\varepsilon = 0$ and $\delta = 2^{-j}$}
\label{fig3}
\end{figure}

\begin{figure}[ht]
\begin{center}
   \rotatebox{0}{\resizebox{!}{0.5\linewidth}{%
   \includegraphics{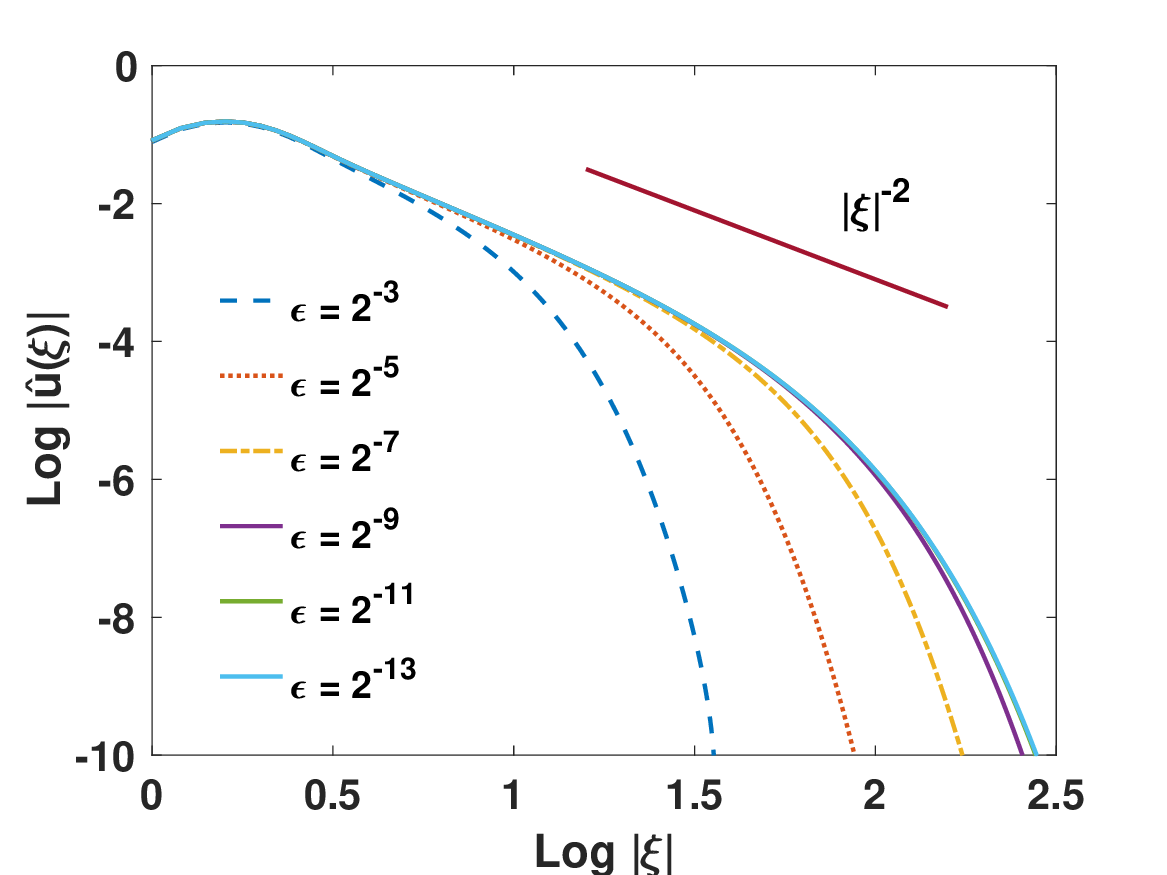}}}
\end{center}
\caption{$\hat u_\varepsilon(\xi)$ in the case $P = 1$, $d = 1$, $\varepsilon = \delta 2^{-j}$ and $\delta = 2^{-10}$}
\label{fig4}
\end{figure}

\bibliographystyle{abbrv}
\bibliography{forcing}
\end{document}